\documentclass[oneside,12pt]{amsart}
\usepackage{amssymb, mathrsfs, amscd}
\usepackage[all]{xy}
\usepackage{dsfont}
\usepackage{graphics}
\usepackage{graphicx}
\usepackage{enumerate}

\newtheorem{theorem}{Theorem}[section]
\newtheorem{definition}[theorem]{Definition}
\newtheorem{lemma}[theorem]{Lemma}

\newtheorem{proposition}[theorem]{Proposition}
\newtheorem{corollary}[theorem]{Corollary}

\newtheorem{remark}[theorem]{Remark}

\newtheorem{notation}[theorem]{Notation}

\author{Dieter Degrijse}
\address{Department of Mathematics, Catholic University of Leuven, Kortrijk, Belgium}%
\email{Dieter.Degrijse@kuleuven-kortrijk.be}%

\author{Nansen Petrosyan}
\address{Department of Mathematics, Catholic University of Leuven, Kortrijk, Belgium}%
\email{Nansen.Petrosyan@kuleuven-kortrijk.be}%
\title{$\mathfrak{F}$-structures and Bredon-Galois cohomology}
\thanks{Both authors were supported by the Research Fund K.U.Leuven.}
\thanks{The second author was also supported by the FWO-Flanders Research Fellowship.}

\date{\today}
\newcommand{\mF}{\mathfrak {F}}
\newcommand{\Z}{\mathbb Z}
\newcommand{\orb}{\mathcal{O}_{\mF}G}
\newcommand{\orbmod}{\mbox{Mod-}\mathcal{O}_{\mF}G}

\newcommand{\nathom}{\mathrm{Hom}_{\mF}}
\newcommand{\orbg}{\mathcal{O}_{\mF}(\Gamma)}

\bigskip
\begin{document}
\maketitle
\begin{abstract}
Let $\mF$ be an arbitrary family of subgroups of a group $G$ and let $\orb$ be the associated orbit category. We investigate interpretations of  low dimensional  $\mF$-Bredon cohomology of $G$ in terms of abelian extensions of $\orb$. Specializing to fixed point functors as coefficients, we derive several group theoretic applications and introduce Bredon-Galois cohomology. We prove an analog of Hilbert's Theorem 90 and show that the second Bredon-Galois cohomology is a certain intersection of relative Brauer groups. As applications, we realize the relative Brauer group $\mathrm{Br}(L/K)$ of a finite separable non-normal extension of fields $L/K$ as a second Bredon cohomology group and show that this approach is quite suitable for finding nonzero elements in $\mathrm{Br}(L/K)$.
\end{abstract}

\tableofcontents
\section{Introduction}
Bredon cohomology was introduced by G. Bredon in \cite{Bredon} as a means to develop an obstruction theory for equivariant extension of maps. It was further developed by Slomi\'{n}ska \cite{slom0}, tom Dieck \cite{tom},  L\"uck \cite{Luck},  Nucinkis \cite{Nucinkis} and many other authors with applications to proper group actions and classifying spaces for families of subgroups (see \cite{Luck2}).

Let $\orbmod$ be the category of contravariant functors from the orbit category $\mathcal{O}_{\mF}G$ to the category of abelian groups $\mathfrak{Ab}$. The Bredon cohomology functors $\mathrm{H}_{\mF}^{\ast}(G,-)$ of a group $G$ with respect to the family of subgroups $\mF$ are the right derived functors of a certain $\mathrm{Hom}$ functor from  $\orbmod$ to $\mathfrak{Ab}$.

Assuming the family $\mF$ is closed under conjugation and finite intersections, one can consider the homotopy category of $G$-$\mathrm{CW}$-complexes with stabilizers in $\mF$. A terminal
object in this category is called a model for the classifying space $E_{\mF}G$, and one can show that these models always exist under the given assumptions on the family (e.g. see \cite{Luck2}). The  augmented cellular chain complex of $E_{\mF}G$ yields a projective resolution which can be used to compute $\mathrm{H}_{\mF}^{\ast}(G,-)$. Partially motivated by the Baum-Connes and Farell-Jones Isomorphism conjectures (see \cite{BCH}, \cite{MV}, \cite{DL}), a lot of attention has gone to questions concerning finiteness properties of $E_{\mF}G$, in particular when $\mF$ is the family of finite subgroups or the family of virtually cyclic subgroups (e.g. see \cite{Kochloukova},\cite{Luck},\cite{LuckWeiermann},\cite{Nucinkis}).

In this paper, we specialize to another aspect of the theory; namely, the interpretation of lower dimensional Bredon cohomology groups in terms of so-called abelian extensions of the orbit category.  This correspondence was first developed by G. Hoff in a more general framework of cohomology of categories (see, \cite{Hoff},\cite{Hoff2}, \cite{Webb}). The cohomology of a category $C$ is defined in a similar way as Bredon cohomology where one just replaces the orbit category by the category $C$. In order to give an interpretation of the second  cohomology of a small category $C$, Hoff considers abelian extensions of $C$. He then derived similar characterizations for the  second cohomology of $C$ as in the case of ordinary group cohomology. We expand the theory developed by Hoff in the case of fixed point functor coefficients.

Let $G$ be a group and let $\mF$ be a collection of subgroups of $G$ that contains the trivial subgroup.
Let $M$ be a left $G$-module. Then the fixed point functor associated to $M$ is the contravariant functor:
\[ \underline{M}: \orb \rightarrow \mathfrak{Ab}: G/H \mapsto M^{\scriptscriptstyle H}, \]
where $M^{\scriptscriptstyle H}$ is the $H$-invariant subgroup of $M$ and  $\underline{M}$ maps a morphism $G/H \xrightarrow{x} G/K$ to $M^{\scriptscriptstyle K} \rightarrow M^{\scriptscriptstyle H} : m \mapsto x\cdot m$.

\begin{definition}\label{def: fstructure }\rm Let $0 \rightarrow M \rightarrow \Gamma \xrightarrow{\pi} G \rightarrow 1$,
be an abelian group extension of $G$ by $M$.
 An {\it $\mF$-structure} on this extension is a collection of subgroups $\{\Gamma_{\scriptscriptstyle H}\}_{H \in \mF}$ of $\Gamma$, such that
\begin{enumerate}[(i)]
 \item for each $H \in \mF$, $\pi: \Gamma_{\scriptscriptstyle H} \xrightarrow{\cong} H$;
 \medskip
 \item for each $H,K \in \mF$ and  $x \in G$, if $x^{\scriptscriptstyle -1}H x\subseteq K$ then there exists  $y \in \pi^{\scriptscriptstyle -1}(x)$ such that $y^{\scriptscriptstyle -1}\Gamma_{\scriptscriptstyle H} y\subseteq \Gamma_{\scriptscriptstyle K}$.
\end{enumerate}
This $\mF$-structure $\{\Gamma_{\scriptscriptstyle H}\}_{H \in \mF}$  is called split if there exists a group homomorphism $s: G \rightarrow \Gamma$ such that $\pi \circ s = id$ and  for each $H \in \mF$, the subgroups $s(H)$ and $\Gamma_{\scriptscriptstyle H}$ are conjugate in $\Gamma$ by $m_{\scriptscriptstyle H}$ for  some $m_{\scriptscriptstyle H} \in M$. Two splittings $s$ and $t$ of an $\mF$-structure on  $(G,M)$
are said to be $M$-conjugate if there exists an element $m \in M$ such that
$s(x)$ and $t(x)$ are conjugate by $m$ for all $x \in G$.
An $\mF$-structure on $(G,M)$ is just an $\mF$-structure on some extension of $G$ by $M$. We define a notion of equivalence of $\mF$-structures and denote the set of equivalence classes of $\mF$-structures on $(G,M)$ by $\mathrm{Str}_{\mF}(G,M)$ (see \ref{def: fstructure}).
\end{definition}
We prove the following.
\begin{theorem} \label{th: main intro2} There is a one-to-one correspondence between $\mathrm{H}^2_{\mF}(G,\underline{M})$ and $\mathrm{Str}_{\mF}(G,M)$ such that the zero element in $\mathrm{H}^2_{\mF}(G,\underline{M})$ corresponds to the class of split $\mF$-structures on $(G,M)$. In particular, $\mathrm{H}^2_{\mF}(G,\underline{M})=0$ if and only if every $\mF$-structure on $(G,M)$ splits. Moreover, there is a one-to-one correspondence between the $M$-conjugacy classes of splittings of the standard split $\mF$-structure on $(G,M)$ and the elements of $\mathrm{H}^1_{\mF}(G,\underline{M})$.  In particular, $\mathrm{H}^1_{\mF}(G,\underline{M})=0$ if and only if all splittings of the standard split $\mF$-structure on $(G,M)$ are $M$-conjugate.
\end{theorem}
Among other applications, this theorem can be used to derive splitting results, as in the following corollary.
\begin{corollary} Let $G$ be either a countable group and $\mathfrak{F}$ the family of finitely generated subgroups of $G$, or let $G$ be a virtually free group with $\mathfrak{F}$ the family of finite subgroups of $G$. Suppose $M \in G\mbox{-mod}$ and consider an abelian extension, $0 \rightarrow M \rightarrow \Gamma \rightarrow G \rightarrow 1$. Then
\begin{enumerate}[(i)]
\item every $\mathfrak{F}$-structure on the extension is split.
\smallskip
\item If $H^1(H,M)=0$ for all $H \in \mathfrak{F}$,  the extension splits if and only if it splits when restricted to every $H \in \mathfrak{F}$.
    \end{enumerate}
\end{corollary}
Our main application of Bredon cohomology with fixed point functor coefficients and its interpretations is related to Galois cohomology. One of the aims of the theory is to investigate a finite Galois extension of fields $E/K$ by studying the cohomology of the Galois group $\mathrm{Gal}(E/K)$ with coefficients in multiplicative group $E^{\times}$ (see \cite{Serre2}). We propose a generalization of Galois cohomology, called \emph{Bredon-Galois cohomology}, as a tool to study  a collection of intermediate fields of a finite Galois extensions $E/K$.
 \begin{definition}\label{def: Bredon-Galois}\rm Let $E/K$ be a finite Galois extension of fields and suppose $\mathcal{F}$ is a collection of intermediate fields. Denote  $G=\mathrm{Gal}(E/K)$ and define the family of subgroups  $\mathfrak{F}=\{H\subseteq G \;|\; E^H=F \mbox{ for some } F\in \mathcal{F}\}$. The {\it $n$-th Bredon-Galois cohomology of $E/K$ for the family $\mathcal{F}$} is by definition $\mathrm{H}^n_{\mF}(G, \underline{E}^{\times})$.
 \end{definition}
 \smallskip
It is well-known that the relative Brauer group $\mathrm{Br}(E/K)$ of a finite Galois extension $E/K$ of fields is isomorphic to $\mathrm{H}^2(\mathrm{Gal}(E/K),E^{\times})$. Utilizing the  correspondence of Theorem \ref{th: main intro2}, we derive a generalization of this result which extends to non-normal extensions.
\begin{theorem} \label{prop: intro bredon galois} Let $E/K$ be a finite Galois extension with Galois group $G$. Let $\mF$ be a family of subgroups of $G$ containing the trivial subgroup, then we have:
\[ \mathrm{H}^2_{\mF}(G,\underline{E}^{\times}) \cong \cap_{H \in \mF} \mathrm{Br}(E^{\scriptscriptstyle H}/K)\subseteq \mathrm{Br}(E/K).\]
In particular, suppose $L/K$ is a finite separable extension and let  $E/L$ be a finite field extension such that $E$ is Galois over $K$. Denote $G=\mathrm{Gal}(E/K)$ and let $H$ be the subgroup of $G$ such that $E^{\scriptscriptstyle H}=L$. Then,
\[
 \mathrm{H}^2_{\mF}(G,\underline{E}^{\times}) \cong \mathrm{Br}(L/K),
\]
where $\mF$ is the smallest family of subgroups of $G$ that is closed under conjugation, closed under taking subgroups and contains $H$.
\end{theorem}
Next, we consider a  finite Galois extension $E/K$ with Galois group $G$.
\begin{theorem} Let $\mF$ be any family of subgroups of $G$ that is closed under conjugation and taking subgroups. For a prime $p$, if the $p$-Sylow subgroups of $G$ are cyclic, then we have a short exact sequence:
\[ 0 \rightarrow  \cap_{H \in \mF} \mathrm{Br}(E^{\scriptscriptstyle H}/K)_{(p)} \rightarrow \mathrm{Br}(E/K)_{(p)} \rightarrow \lim_{H \in \mF}\mathrm{Br}(E/E^{\scriptscriptstyle H})_{(p)} \rightarrow 0. \]
\end{theorem}

In \cite{FeinSaltmanSchacher} and \cite{FeinSchacher}, Fein, Saltman, and Schacher used  Galois cohomology to
obtain interesting structural results on relative Brauer groups of finite separable extensions of fields.  Our results indicate that  Bredon-Galois cohomology is perhaps  a better suited tool for studying relative Brauer groups of such extensions.

\begin{definition} \rm Let $k \in \mathbb{N}\cup \{\infty\}$. We say that an extension of fields $E/K$ is $k$-\emph{admissible} if it is a finite Galois extension with Galois group $G$ such that for every nontrivial $\sigma \in G$, there exist at least $k$ discrete valuations $\pi$ of $E$ such that $G_{\pi}=\langle\sigma\rangle$ and  $\pi|_{K}$ is unramified in $E$.
\end{definition}

\begin{theorem}
Let $E/K$ be a $k$-admissible extension with Galois group $G$ and let $\mF$ be any family of subgroups of $G$  closed under conjugation and taking subgroups. If there exists an element $\sigma \in G$ such that the group:
\[  \{f \in \mathrm{Hom}(\langle\sigma\rangle,\mathbb{Q}/\mathbb{Z}) \; | \; f(H\cap \langle\sigma\rangle)=0, \ \forall \ H \in \mF \}
\]
 is nontrivial, then $\cap_{H \in \mF} \mathrm{Br}(E^{\scriptscriptstyle H}/K)$ contains at least $k$ non-zero elements.
\end{theorem}
As a corollary, we reprove a result first obtained in \cite{FeinKantorSchacher} which states that for a finite separable extension of global fields, the relative Brauer group is infinite.
\section{Notations and preliminaries}
Let $G$ be a group and let $\mathfrak{F}$ be an arbitrary collection of subgroups of $G$. We refer to $\mathfrak{F}$ as a \emph{family of subgroups} of $G$. The groups denoted by $H,K$ or $L$ are always contained in $\mathfrak{F}$.
The \emph{orbit category} $\orb$ is a category defined as follows: the objects are the left coset spaces $G/H$ for all $H \in \mathfrak{F}$ and the morphisms are all $G$-equivariant maps between the objects. Note that every morphism $\varphi: G/H \rightarrow G/K$ is completely determined by $\varphi(H)$, since $\varphi(xH)=x\varphi(H)$ for all $x \in G$. Moreover, there exists a morphism: $$G/H \rightarrow G/K : H \mapsto xK \mbox{ if and only if } x^{\scriptscriptstyle -1}Hx \subseteq K.$$ We denote the morphism $\varphi: G/H \rightarrow G/K: H\mapsto xK$  by $G/H \xrightarrow{x} G/K$. Viewing $G/K$ as a $G$-set, we write $(G/K)^{\scriptscriptstyle H}$ for the invariants of $G/K$ under the action of $H$. There is a bijective correspondence between $\mathrm{Mor}(G/H,G/K)$ and $(G/K)^{\scriptscriptstyle H}$, which is obtained by mapping a morphism $G/H \xrightarrow{x} G/K$ from $\mathrm{Mor}(G/H,G/K)$ to $xK \in (G/K)^{\scriptscriptstyle H}$.

An \emph{$\orb$-module} is a contravariant functor $M: \orb \rightarrow \mathfrak{Ab}$. The \emph{category of $\orb$-modules} is denoted by $\orbmod$ and is defined as follows: the objects are all the $\orb$-modules, the morphisms are all the natural transformations between the objects. Let $M$ be a left $G$-module. The {\it fixed point functor} associated to $M$ is  defined as the contravariant functor:
\[ \underline{M}: \orb \rightarrow \mathfrak{Ab}: G/H \mapsto M^{\scriptscriptstyle H} \]
where $M^{\scriptscriptstyle H}$ is the $H$-invariant subgroup of $M$, and  $\underline{M}$ maps a morphism $G/H \xrightarrow{x} G/K$ to $M^{\scriptscriptstyle K} \rightarrow M^{\scriptscriptstyle H} : m \mapsto x\cdot m$.
To do homological algebra in $\orbmod$ and in particular to define Bredon cohomology, we follow the approach of L\"{u}ck in \cite{Luck}, Chapter $9$. Let us briefly recall some of the elementary notions we will need.

A sequence
$0\rightarrow M' \rightarrow M \rightarrow M'' \rightarrow 0$
in $\orbmod$ is called {\it exact} if it is exact after evaluating in $G/H$, for all $H \in \mathfrak{F}$. Take $M \in \orbmod$ and consider the left exact functor:
\[ \nathom(M,-) : \orbmod \rightarrow \mathfrak{Ab}: N \mapsto \nathom(M,N), \]
where $\nathom(M,N)$ is the abelian group of all natural transformations from $M$ to $N$. Then
$M$ is a projective $\orb$-module if and only if this functor is exact. By considering the contravariant $\mathrm{Hom}_{\mathfrak{F}}$-functor, one can define injective $\orb$-modules in a similar way. It can be shown that $\orbmod$ contains enough projective and injective objects to construct projective and injective resolutions. Hence, one can construct bi-functors $\mathrm{Ext}^{n}_{\orb}(-,-)$ that have all the usual properties. The \emph{$n$-th Bredon cohomology of $G$} with coefficients $M \in \orbmod$ is by definition:
\[ \mathrm{H}^n_{\mathfrak{F}}(G,M)= \mathrm{Ext}^{n}_{\orb}(\underline{\mathbb{Z}},M), \]
where ${\Z}$ is a trivial $G$-module.

Finally, we describe the free modules in $\orbmod$. An \emph{$\mathfrak{F}$-set} is a collection of (possibly empty) sets $\{\Delta_{\scriptscriptstyle H}  \ | \ H \in \mathfrak{F}\}$. By setting $\Delta= \coprod_{H \in \mathfrak{F}}\Delta_{\scriptscriptstyle H}$, we obtain a set map $\beta: \Delta \rightarrow \mathfrak{F}: \delta \in \Delta_{\scriptscriptstyle H} \mapsto H$. We call $\beta$ a \emph{base function} for $\mathfrak{F}$. Note that $\beta^{\scriptscriptstyle -1}(H)=\Delta_{\scriptscriptstyle H}$ for each $H \in \mathfrak{F}$ and that each base function of $\mathfrak{F}$ uniquely determines an $\mathfrak{F}$-set and vice versa. In the future, we will denote an $\mathfrak{F}$-set by the pair $(\Delta,\beta)$, where $\beta$ is the defining base function from the set $\Delta$ to the family $\mathfrak{F}$. A map of $\mathfrak{F}$-sets from $(\Delta,\beta)$ to $(\Delta',\beta')$ is a set map $f$ from $\Delta$ to $\Delta'$ such that $\beta' \circ f = \beta$. Alternatively, a map of $\mathfrak{F}$-sets from $(\Delta,\beta)$ to $(\Delta',\beta')$ can be viewed as a collection of set maps $\{f_{\scriptscriptstyle H} : \Delta_{\scriptscriptstyle H} \rightarrow \Delta'_{\scriptscriptstyle H}  \ | \ H \in \mathfrak{F}\}$.

Note that any $\orb$-module $M$ has a structure of an $\mathfrak{F}$-set. By using the forgetful functor, define $\Delta_{\scriptscriptstyle K}= M(G/K)$  for each $K \in \mathfrak{F}$. We denote the $\mathfrak{F}$-set associated to $M$ by $\mathfrak{F}(M)$. Obviously, a natural transformation $\theta: M \rightarrow N$ induces a map of $\mathfrak{F}$-sets $\theta: \mathfrak{F}(M) \rightarrow \mathfrak{F}(N)$. A base function $\beta: \Delta \rightarrow \mathfrak{F}$ is said to \emph{belong} to an $\orb$-module $M$ if there exist an injective map of $\mathfrak{F}$-sets $i: (\Delta,\beta) \rightarrow \mathfrak{F}(M)$ (notation $\beta \subset_{i} M$).

An $\orb$-module $F$ is called \emph{free on the $\mathfrak{F}$-set $(\Delta,\beta)$} if $\beta \subset_i M$ and if for each $M \in \orbmod$, every map of $\mathfrak{F}$-sets $f: (\Delta,\beta) \rightarrow \mathfrak{F}(M)$ can be uniquely extended to a natural transformation $\bar{f}: F \rightarrow M$, i.e. there exists a unique natural transformation $\bar{f}: F \rightarrow M$ such that the diagram of maps of $\mathfrak{F}$-sets:
\[ \xymatrix{   \mathfrak{F}(F) \ar[r]^{\bar{f}} & \mathfrak{F}(M) \\
                (\Delta,\beta) \ar[u]^{i} \ar[ur]^{f} & } \]
commutes. In general, an $\orb$-module $F$ is called \emph{free} if it is free on some $\mathfrak{F}$-set $(\Delta,\beta)$.

It is not difficult to show that free $\orb$-modules are projective,  direct sums of free $\orb$-modules are free, and an $\orb$-module is projective if and only if it is a direct summand of a free $\orb$-module.

Now, take $K \in \mathfrak{F}$ and define the $\orb$-module:
\[ \mathbb{Z}[-,G/K]: \orb \rightarrow \mathfrak{Ab} : G/H \rightarrow \mathbb{Z}[G/H,G/K], \]
where $\mathbb{Z}[G/H,G/K]$ is the free $\mathbb{Z}$-module generated by $\mathrm{Mor}(G/H,G/K)$, and for each $\varphi \in \mathrm{Mor}(G/H_1,G/H_2)$, the morphism: $$\mathbb{Z}[\varphi,G/K]: \mathbb{Z}[G/H_2,G/K] \rightarrow \mathbb{Z}[G/H_1,G/K]$$ is defined by  $\mathbb{Z}[\varphi,G/K](\alpha)= \alpha \circ \varphi$ for all $\alpha \in \mathrm{Mor}(G/H_2,G/K)$. One can check that $\mathbb{Z}[-,G/K]$ is a free $\orb$-module and $\bigoplus_{\delta \in \Delta}\mathbb{Z}[-,G/\beta(\delta)]$ is a model for the free $\orb$-module on the $\mathfrak{F}$-set $(\Delta,\beta)$.

\section{Abelian Extensions of the Orbit Category}
In Section \ref{sec: ab ext gen}, we recall and rephrase some terminology and results from \cite{Hoff} and \cite{Hoff2}. In Section \ref{sec: ab ext fixed}, we turn to fixed point functors and derive the main results relating $\mathfrak{F}$-structures and abelian extensions of the orbit category.
\subsection{Definition and Basic Properties} \label{sec: ab ext gen}
Let $G$ be a group and let $\mathfrak{F}$ be a family of subgroups of $G$. Let $\{M(G/H)\}_{H \in \mathfrak{F}}$ be a collection of $\mathbb{Z}$-modules indexed over the family $\mathfrak{F}$. We define the category $\mathfrak{M}$ as follows. The objects of $\mathfrak{M}$ are precisely the $\mathbb{Z}$-modules $M(G/H)$ for all $H \in \mathfrak{F}$, the set of morphisms $\mathrm{Mor}(M(G/H),M(G/K))$ is empty if $H \neq K$ and
$$\mathrm{Mor}(M(G/H),M(G/H))=\{\bar{m}_{\scriptscriptstyle H} : M(G/H) \rightarrow M(G/H): m \mapsto m_{\scriptscriptstyle H} + m \;|\;m_{\scriptscriptstyle H} \in M(G/H)\}.$$
We call $\mathfrak{M}$ the category associated to the collection of $\mathbb{Z}$-modules $\{M(G/H)\}_{H \in \mathfrak{F}}$.
For all $n_{\scriptscriptstyle H},m_{\scriptscriptstyle H} \in M(G/H)$, we define $\bar{m}_{\scriptscriptstyle H}\pm\bar{n}_{\scriptscriptstyle H}$ to be the morphism $\overline{m_{\scriptscriptstyle H} \pm n_{\scriptscriptstyle H}}$.

\begin{definition} \rm An  \emph{abelian right extension of $\orb$ by $\mathfrak{M}$} consists of a category $\mathfrak{E}$ with objects $\{E(H)\}_{H \in \mathfrak{F}}$ and functors:
\[  \mathfrak{M} \xrightarrow{i} \mathfrak{E} \xrightarrow{\pi} \orb \]
satisfying the following properties.
\begin{enumerate}[(i)]
\item $i(M(G/H))= E(H)$ and $\pi(E(H))=G/H$ for all $H \in \mathfrak{F}$;
\medskip
\item $\forall H \in \mathfrak{F}, \forall \bar{m}_{\scriptscriptstyle H},\bar{n}_{\scriptscriptstyle H} \in \mathrm{Mor}({M}(G/H),{M}(G/H))$, if $i(\bar{m}_{\scriptscriptstyle H})=i(\bar{n}_{\scriptscriptstyle H})$, then $\bar{m}_{\scriptscriptstyle H}=\bar{n}_{\scriptscriptstyle H}$, i.e. $i$ is injective on morphisms;
    \medskip
\item $\forall H,K \in \mathfrak{F}, \forall \varphi \in \mathrm{Mor}(G/H,G/K)$, there exists a $\theta \in \mathrm{Mor}(E(H),E(K))$ such that $\pi(\theta)=\varphi$, i.e. $\pi$ is surjective on morphisms;
    \medskip
\item $\pi(i(\bar{m}_{\scriptscriptstyle H}))=id$ for all  $H \in \mathfrak{F}$ and for all $\bar{m}_{\scriptscriptstyle H} \in \mathrm{Mor}(M(G/H),M(G/H))$;
\medskip
\item $\forall H,K \in \mathfrak{F}, \forall \varphi_1,\varphi_2 \in \mathrm{Mor}({E}(H),{E}(K))$: if $\pi(\varphi_1)=\pi(\varphi_2)$, then there exists a unique $\bar{m}_{\scriptscriptstyle H} \in \mathrm{Mor}(M(G/H),M(G/H))$ such that $\varphi_1 \circ i(\bar{m}_{\scriptscriptstyle H})= \varphi_2$.
\end{enumerate}
\end{definition}
\smallskip
\begin{remark} \rm {$\;$}
\begin{enumerate}[(a)]
\item It follows from property (v) that  $\forall H,K \in \mathfrak{F}, \forall \varphi \in \mathrm{Mor}({E}(H),{E}(K))$, $\varphi \circ i(\bar{m}_{\scriptscriptstyle H}) = \varphi$ implies  $m_{\scriptscriptstyle H}=0\in M(G/H)$.
    \smallskip
\item Similarly, one can define an abelian left extension of $\orb$ by $\mathfrak{M}$ by replacing property (v) with the following: $\forall H,K \in \mathfrak{F}, \forall \varphi_1,\varphi_2 \in \mathrm{Mor}({E}(H),E(K))$, if $\pi(\varphi_1)=\pi(\varphi_2)$, then there exists a unique $\bar{m}_{\scriptscriptstyle K} \in \mathrm{Mor}(M(G/K),M(G/K))$ such that $i(\bar{m}_{\scriptscriptstyle K}) \circ \varphi_1 = \varphi_2$.
    \smallskip
\item In \cite{Xu}, Xu constructs a spectral sequence associated to an extension of categories, which generalizes the Lyndon-Hochschild-Serre spectral sequence.
\end{enumerate}
\end{remark}
\begin{definition} \rm The abelian right $\mathfrak{M}$-extensions, $\mathfrak{M} \xrightarrow{i_1} \mathfrak{E}_1 \xrightarrow{\pi_1} \orb$ and $\mathfrak{M} \xrightarrow{i_2} \mathfrak{E}_2 \xrightarrow{\pi_2} \orb$, are called \emph{equivalent} if there exists a functor $\theta : \mathfrak{E}_1 \rightarrow \mathfrak{E}_2$ such that the following diagram commutes:
\[ \xymatrix{  \mathfrak{M} \ar[r]^{i_1} \ar[d]^{id} & \mathfrak{E}_1 \ar[r]^{\pi_1} \ar[d]^{\theta} & \orb \ar[d]^{id}   \\
 \mathfrak{M} \ar[r]^{i_2}  & \mathfrak{E}_2 \ar[r]^{\pi_2}  & \orb . } \]
 \smallskip
It is not difficult to check that this implies $\theta({E}_1(H))={E}_2(H)$ for all $H \in \mathfrak{F}$ and that $\theta$ is bijective on morphisms.
\end{definition}
It is a standard fact that a group extension of a group $G$ by an abelian group $A$ induces a $G$-module structure on $A$. The following proposition shows that an analogous result holds in the case of abelian extensions of orbit categories.
\begin{proposition} An abelian right extension of $\orb$ by $\mathfrak{M}$ defines an $\orb$-module $M : \orb \rightarrow \mathfrak{Ab}: G/K \mapsto M(G/K)$  such that equivalent extensions give rise to the same $\orb$-module structure on $M$.
\end{proposition}
\begin{proof}
Let $\mathfrak{M} \xrightarrow{i_1} \mathfrak{E}_1 \xrightarrow{\pi_1} \orb$ be an abelian right extension of $\orb$ by $\mathfrak{M}$. To make $M : \orb \rightarrow \mathfrak{Ab}: G/K \mapsto M(G/K)$ into a contravariant functor, we first need to define a $\mathbb{Z}$-module homomorphism $M(\varphi): M(G/K) \rightarrow M(G/H)$ for every morphism $\varphi \in \mathrm{Mor}(G/H,G/K)$.

Let $\varphi \in \mathrm{Mor}(G/H,G/K)$. By surjectivity of $\pi_1$, we can find a $\psi \in \mathrm{Mor}(E_1(H),E_1(K))$ such that $\pi_1(\psi)=\varphi$. Now, for any given $m_{\scriptscriptstyle K} \in M(G/K)$, we have $\pi_1(i_1(\bar{m}_{\scriptscriptstyle K})\circ \psi)=\varphi$. Hence, there exists a unique $\bar{m}_{\scriptscriptstyle H} \in \mathrm{Mor}(M(G/H),M(G/H))$ such that $i_1(\bar{m}_{\scriptscriptstyle K})\circ \psi= \psi \circ i_1(\bar{m}_{\scriptscriptstyle H})$. We define:
\[ M(\varphi): M(G/K) \rightarrow M(G/H): m_{\scriptscriptstyle K} \mapsto m_{\scriptscriptstyle H}. \]
Repeated use of the uniqueness and injectivity properties shows that this is a well-defined (independent of the choice $\psi$) homomorphism of $\mathbb{Z}$-modules that turns $M$ into a contravariant functor.

It remains to show that equivalent extensions give rise to the same $\orb$-module structure on $M$. Let $\mathfrak{M} \xrightarrow{i_2} \mathfrak{E}_2 \xrightarrow{\pi_2} \orb$ be an extension that is equivalent to $\mathfrak{M} \xrightarrow{i_1} \mathfrak{E}_1 \xrightarrow{\pi_1} \orb$ via $\theta: \mathfrak{E}_1 \rightarrow \mathfrak{E}_2$. Let $\varphi \in \mathrm{Mor}(G/H,G/K)$ and $\psi \in \mathrm{Mor}(E_1(H),E_1(K))$ be such that $\pi_1(\psi)=\varphi$. Now, for an element $m_{\scriptscriptstyle K} \in M(G/K)$, let $\bar{m}_{\scriptscriptstyle H}$ be the unique element in $\mathrm{Mor}(M(G/H),M(G/H))$ such that $i_1(\bar{m}_{\scriptscriptstyle K})\circ \psi= \psi\circ i_1(\bar{m}_{\scriptscriptstyle H})$. Applying $\theta$ to this equation, we obtain $i_2(\bar{m}_{\scriptscriptstyle K})\circ \theta(\psi)=\theta(\psi)\circ i_2(\bar{m}_{\scriptscriptstyle H})$. Since $\pi_2(\theta(\psi))=\varphi$, we conclude that the $\orb$-module structure on $M$ described above is the same for equivalent extensions.
\end{proof}
\begin{remark} \rm Similarly one can show that abelian left extensions give rise to covariant functors $M: \orb \rightarrow \mathfrak{Ab}$.
\end{remark}
The proposition allows us to state the following definition.
\begin{definition} \rm Let $M$ be a $\orb$-module and let $\mathfrak{M}$ be the category associated to $\{M(G/H)\}_{H \in \mathfrak{F}}$. We denote by $\mathrm{Ext}_{\mathfrak{F}}(G,M)$ the set of equivalence classes of abelian right extensions of $\orb$ by $\mathfrak{M}$, such that induced the $\orb$-module structure on $M$ by the extensions is the given one. A representative of an element in $\mathrm{Ext}_{\mathfrak{F}}(G,M)$ is then called \emph{an abelian right extension of $\orb$ by $M$}.
\end{definition}
Next, we describe the notion of a split extension.
\begin{definition} \rm An abelian right extension $\mathfrak{M} \xrightarrow{i} \mathfrak{E} \xrightarrow{\pi} \orb$ is called \emph{split} if there exists a functor $s: \orb \rightarrow \mathfrak{E}$ such that $\pi \circ s = id$.
\end{definition}
We will now construct, for any $\orb$-module $M$, a split abelian right extension of $\orb$ by $M$ and show that any split abelian right extension of $\orb$ by $M$ is equivalent to it. This will allow us to consider  the class of split extensions inside $\mathrm{Ext}_{\mathfrak{F}}(G,M)$. \\
\indent Let $M$ be a $\orb$-module and let $\mathfrak{M}$ be the category associated to $\{M(G/H)\}_{H \in \mathfrak{F}}$. One defines $\mathfrak{M}\rtimes \orb$ to be the category with objects $M(G/H)\times G/H $ (set product) for each $H \in \mathfrak{F}$. The morphisms of $\mathfrak{M}\rtimes \orb$ are: for each $\varphi \in \mathrm{Mor}(G/H,G/K)$ and each $\bar{m}_{\scriptscriptstyle H} \in \mathrm{Mor}(M(G/H),M(G/H))$ we have a morphism: $$(\bar{m}_{\scriptscriptstyle H},\varphi) \in \mathrm{Mor}(M(G/H)\times G/H,M(G/K)\times G/K).$$
Note that $(\bar{m}_{\scriptscriptstyle H},\varphi)$ is not a set map from $M(G/H)\times G/H$ to $M(G/K)\times G/K$, it is just a morphism in the formal categorical sense. The composition $(\bar{m}_{\scriptscriptstyle K},\psi)\circ (\bar{m}_{\scriptscriptstyle H},\phi)$ of two morphisms: $$(\bar{m}_{\scriptscriptstyle H},\varphi) \in \mathrm{Mor}(M(G/H)\times G/H,M(G/K)\times G/K)$$ and $$(\bar{m}_{\scriptscriptstyle K},\psi) \in \mathrm{Mor}(M(G/K)\times G/K,M(G/L)\times G/L)$$\medskip is defined to be: $$ (\bar{m}_{\scriptscriptstyle H}+\overline{M(\varphi)(m_{\scriptscriptstyle K})},\psi \circ \varphi) \in  \mathrm{Mor}(M(G/H)\times G/H,M(G/L)\times G/L).$$ It is not difficult to see that $\mathfrak{M}\rtimes \orb$ is a category. Now, we define the functor $i: \mathfrak{M} \rightarrow \mathfrak{M}\rtimes \orb$ by setting: $$i(M(G/H))=M(G/H)\times G/H \mbox{\; and \;} i(\bar{m}_{\scriptscriptstyle H})=(\bar{m}_{\scriptscriptstyle H},id)$$
\medskip
for all $\bar{m}_{\scriptscriptstyle H} \in \mathrm{Mor}(M(G/H),M(G/H))$ and for all $H \in \mathfrak{F}$. Clearly, the functor $i$ is injective on morphisms. Define the functor $\pi: \mathfrak{M}\rtimes \orb \rightarrow \orb$ by  $\pi(M(G/H)\times G/H)= G/H$ and $\pi(\bar{m}_{\scriptscriptstyle H},\varphi)=\varphi$ for all $$(\bar{m}_{\scriptscriptstyle H},\varphi) \in \mathrm{Mor}(M(G/H)\times G/H,M(G/K)\times G/K)$$
\medskip
and for all $H,K \in \mathfrak{F}$. Clearly, the functor $\pi$ is surjective on morphisms.

\begin{proposition}\label{lemma: unique split ext} Let $M$ be an $\orb$-module, then
\[ \mathfrak{M} \xrightarrow{i} \mathfrak{M} \rtimes \orb \xrightarrow{\pi} \orb \]
is a split abelian right extension of $\orb$ by $M$ and it is equivalent to any split abelian right extension of $\orb$ by $M$.
\end{proposition}
\begin{proof}
Clearly, $\pi \circ i(\bar{m}_{\scriptscriptstyle H})=id$ for all morphisms $\bar{m}_{\scriptscriptstyle H}$. Now, let us check property (v) of the definition of right abelian extensions. If two morphisms in $$\mathrm{Mor}(M(G/H)\times G/H,M(G/K)\times G/K)$$ are mapped to the same morphism under $\pi$, this implies that these morphism are of the form $(\bar{m}_{\scriptscriptstyle H},\varphi)$ and $(\bar{n}_{\scriptscriptstyle H},\varphi)$. A quick computation then shows that $(\bar{m}_{\scriptscriptstyle H},\varphi) = (\bar{n}_{\scriptscriptstyle H},\varphi) \circ i(\bar{m}_{\scriptscriptstyle H}-\bar{n}_{\scriptscriptstyle H})$. This resolves the existence part of (v).

Next, suppose that  $\bar{m} \in \mathrm{Mor}(M(G/H),M(G/H))$ such that
$(\bar{m}_{\scriptscriptstyle H},\varphi)= (\bar{n}_{\scriptscriptstyle H},\varphi) \circ i(\bar{m})$. It follows that $(\bar{n}_{\scriptscriptstyle H}+\bar{m},\varphi)=(\bar{n}_{\scriptscriptstyle H}+ (\bar{m}_{\scriptscriptstyle H}-\bar{n}_{\scriptscriptstyle H}),\varphi)$, hence $\bar{m}=\bar{m}_{\scriptscriptstyle H}-\bar{n}_{\scriptscriptstyle H}$ which shows the uniqueness part of (v).

Since $i(\bar{m}_{\scriptscriptstyle K}) \circ (\bar{0},\varphi)= (\bar{0},\varphi) \circ i(\overline{M(\varphi)({m}_{\scriptscriptstyle K})})$ for each $\varphi \in \mathrm{Mor}(G/H,G/K)$ and each $m_{\scriptscriptstyle K} \in M(G/K)$, it follows that
$\mathfrak{M} \xrightarrow{i} \mathfrak{M} \rtimes \orb \xrightarrow{\pi} \orb$
is a abelian right extension of $\orb$ by $M$. This extension is split because the functor $s: \orb \rightarrow \mathfrak{M}\rtimes \orb$ defined by, $s(G/H)=M(G/H)\times G/H$ and s$(\varphi)=(\bar{0},\varphi)$ satisfies $\pi \circ s=id$.

Now, let $\mathfrak{M} \xrightarrow{i'} \mathfrak{E} \xrightarrow{\pi'} \orb$ be a split abelian right extension of $\orb$ by $M$, with splitting functor $s': \orb \rightarrow \mathfrak{E}$. To show that this extension is equivalent to $\mathfrak{M} \xrightarrow{i} \mathfrak{M} \rtimes \orb \xrightarrow{\pi} \orb$,
we need to construct a functor $\theta: \mathfrak{M} \rtimes \orb \rightarrow \mathfrak{E}$ such that the diagram:
\[ \xymatrix{  \mathfrak{M} \ar[r]^{i} \ar[d]^{id} &  \mathfrak{M} \rtimes \orb  \ar[r]^{\pi} \ar[d]^{\theta} & \orb \ar[d]^{id}   \\
 \mathfrak{M} \ar[r]^{i'}  & \mathfrak{E} \ar[r]^{\pi'}  & \orb , } \]
 commutes.
We define the functor $\theta$ as follows: $\theta(M(G/H)\times G/H)=E(H)$ and $\theta((\bar{m}_{\scriptscriptstyle H},\varphi))$ \\$=s'(\varphi)\circ i'(\bar{m}_{\scriptscriptstyle H})$ for each $$(\bar{m}_{\scriptscriptstyle H},\varphi) \in \mathrm{Mor}(M(G/H)\times G/H,M(G/K)\times G/K)$$ and each $H,K \in \mathfrak{F}$. We need to verify that $\theta$ is indeed a functor. Clearly, we have $\theta(id)=id$. Consider the morphisms: $$(\bar{m}_{\scriptscriptstyle H},\varphi) \in \mathrm{Mor}(M(G/H)\times G/H,M(G/K)\times G/K)$$ and $$(\bar{m}_{\scriptscriptstyle K},\phi) \in \mathrm{Mor}(M(G/K)\times G/K,M(G/L)\times G/L).$$ Their composition is given by: $$(\bar{m}_{\scriptscriptstyle H}+\overline{M(\varphi)(m_{\scriptscriptstyle K})},\phi \circ \varphi) \in  \mathrm{Mor}(M(G/H)\times G/H,M(G/L)\times G/L).$$ Using the fact that $M(\varphi)(m_{\scriptscriptstyle K})$ is the unique element such that $i'(\bar{m}_{\scriptscriptstyle K})\circ s'(\varphi)=s'(\varphi) \circ i'(\overline{M(\varphi)(m_{\scriptscriptstyle K})})$, we compute:
\begin{eqnarray*}
s'(\phi)\circ i'(\bar{m}_{\scriptscriptstyle K})\circ s'(\varphi)\circ i'(\bar{m}_{\scriptscriptstyle H}) &=&  s'(\phi) \circ s'(\varphi) \circ i'(\overline{M(\varphi)({m}_{\scriptscriptstyle K})})\circ i'(\bar{m}_{\scriptscriptstyle H}) \\
&=&  s'(\phi \circ \varphi) \circ i'(\overline{M(\varphi)({m}_{\scriptscriptstyle K})}+ \bar{m}_{\scriptscriptstyle H}).
\end{eqnarray*}
Hence, $\theta$ is a functor. The fact that $\theta$  makes the diagram commute follows immediately.
\end{proof}
The split abelian right extension constructed in this proposition will be referred to as the \emph{standard split abelian right extension} of $\orb$ by $M$.
\subsection{Extension of the orbit category by a fixed point functor} \label{sec: ab ext fixed}
Let $G$ be a group, let $\mathfrak{F}$ be a family of subgroups that contains the trivial subgroup and take $M \in G\mbox{-mod} $.
Suppose we have an abelian group extension:
\begin{equation} \label{eq: group extension} 0 \rightarrow M \rightarrow \Gamma \xrightarrow{\pi} G \rightarrow 1. \end{equation}
Note that this implies that the $G$-module structure on $M$ induced by the extension coincides with the original one, i.e. $\pi(x) \cdot m = xmx^{\scriptscriptstyle -1}$ for all $x \in \Gamma$ and all $m \in M$.
\begin{definition}\label{def: fstructure} \rm An \emph{$\mathfrak{F}$-structure} on (\ref{eq: group extension}) is a collection of subgroups $\Gamma_{\scriptscriptstyle H}$ of $\Gamma$ for each $H \in \mathfrak{F}$, such that:
\begin{enumerate}[(i)]
 \item for each $H \in \mathfrak{F}$, $\pi: \Gamma_{\scriptscriptstyle H} \xrightarrow{\cong} H$;
 \medskip
 \item for each $H,K \in \mathfrak{F}$ and $x \in G$, if $x^{\scriptscriptstyle -1}H x\subseteq K$, then there exists  $y \in \pi^{\scriptscriptstyle -1}(x)$ such that $y^{\scriptscriptstyle -1}\Gamma_{\scriptscriptstyle H} y\subseteq \Gamma_{\scriptscriptstyle K}$.
\end{enumerate}
This collection of subgroups of $\Gamma$ is also denoted by $\mathfrak{F}$. It readily follows that $\Gamma_{\{e\}}=\{e\}$, $\Gamma_{\scriptscriptstyle H} \cap M=\{e\}$ and
that $\Gamma_{\scriptscriptstyle H}$ commutes with $M^{\scriptscriptstyle H}$ for each $H \in \mathfrak{F}$. \\ \\
Now, let us suppose we have another group extension:
\begin{equation} \label{eq: group extension2} 0 \rightarrow M \rightarrow \Gamma' \xrightarrow{\pi'} G \rightarrow 1. \end{equation}
The $\mathfrak{F}$-structures $\{\Gamma_{\scriptscriptstyle H}\}_{H \in \mathfrak{F}}$ on $(\ref{eq: group extension})$ and $\{\Gamma_{\scriptscriptstyle H}'\}_{H \in \mathfrak{F}}$ on $(\ref{eq: group extension2})$ are called \emph{equivalent} if and only if there exists a group homomorphism $\theta: \Gamma \rightarrow \Gamma'$ such that
\[ \xymatrix{  0 \ar[r] & M \ar[r] \ar[d]^{id} & \Gamma \ar[r]^{\pi} \ar[d]^{\theta} & G \ar[d]^{id}  \ar[r]& 1   \\
 0 \ar[r] & M \ar[r]  & \Gamma' \ar[r]^{\pi'}  & G \ar[r] & 1 } \]
commutes and such that for each $H \in \mathfrak{F}$, $\theta(\Gamma_{\scriptscriptstyle H})=m_{\scriptscriptstyle H}^{\scriptscriptstyle -1}\Gamma_{\scriptscriptstyle H}'m_{\scriptscriptstyle H}$ for some $m_{\scriptscriptstyle H} \in M$ (note that $m_{\scriptscriptstyle H}$ depends on $H$). It is clear that, in this case, $\theta$ is an isomorphism and $(\ref{eq: group extension})$ is equivalent with $(\ref{eq: group extension2})$ in the usual sense of group extension.

An $\mathfrak{F}$-structure $\{\Gamma_{\scriptscriptstyle H}\}_{H \in \mathfrak{F}}$ on $(\ref{eq: group extension})$ is called \emph{split} if there exists a group homomorphism $s: G \rightarrow \Gamma$ such that $\pi \circ s = id$ and such that for each $H \in \mathfrak{F}$, $s(H)=m_{\scriptscriptstyle H}^{\scriptscriptstyle -1}\Gamma_{\scriptscriptstyle H}m_{\scriptscriptstyle H}$ for some $m_{\scriptscriptstyle H} \in M$ ($m_{\scriptscriptstyle H}$ depends on $H$). This implies that $(\ref{eq: group extension})$ is split extension in the usual sense. Note that, in general, a split extension of $G$ by $M$ can have a non-split $\mathfrak{F}$-structure.

An \emph{$\mathfrak{F}$-structure on $(G,M)$} is just an $\mathfrak{F}$-structure on some abelian extension of $G$ by $M$.
By $\mathrm{Str}_{\mathfrak{F}}(G,M)$, we denote the set of all equivalence classes of $\mathfrak{F}$-structures on $(G,M)$.
\end{definition}
Now, let us consider the standard split extension:
\[ 0 \rightarrow M \rightarrow \Gamma = M \rtimes G \rightarrow G \rightarrow 1 \]
and define $\Gamma_{\scriptscriptstyle H}=(0,H)$ for each $H \in \mathfrak{F}$. One easily verifies that $\{\Gamma_{\scriptscriptstyle H}\}_{H \in \mathfrak{F}}$ defines a split $\mathfrak{F}$-structure on the standard split extension of $G$ by $M$. We call this the {\it standard split $\mathfrak{F}$-structure} on $(G,M)$. It is not difficult to check that every split $\mathfrak{F}$-structure on $(G,M)$ is equivalent to the standard split $\mathfrak{F}$-structure on $(G,M)$. It follows that we can consider the equivalence class of all split $\mathfrak{F}$-structures in $\mathrm{Str}_{\mathfrak{F}}(G,M)$.

It turns out that, in certain cases, we can view $\mathrm{Str}_{\mathfrak{F}}(G,M)$ as a subspace of $\mathrm{Ext}(G,M)$, the set of equivalence classes of extensions of $G$ by $M$.
\begin{proposition} \label{prop: ker ext} Let $\overline{\mathfrak{F}}$ be the smallest family of subgroups of $G$ containing $\mF$ that is  closed under conjugation.  Suppose $\mathrm{H}^1(H,M)=0$ for every $H \in \overline{\mathfrak{F}}$. Then property (ii) of the definition of $\mathfrak{F}$-structure follows from property (i). Furthermore, $\mathrm{Str}_{\mathfrak{F}}(G,M)$ injects onto a subset of $\mathrm{Ext}(G,M)$ such that under the bijective correspondence between $\mathrm{Ext}(G,M)$ and $\mathrm{H}^2(G,M)$, the set $\mathrm{Str}_{\mathfrak{F}}(G,M)$ corresponds to $$\cap_{H \in \mathfrak{F}}\mathrm{Ker} \Big( i^2_{\scriptscriptstyle H}: \mathrm{H}^2(G,M) \rightarrow \mathrm{H}^2(H,M)\Big)$$ where $i^2_{\scriptscriptstyle H}$ is induced by the inclusion map $i_{\scriptscriptstyle H} : H \hookrightarrow G$.
\end{proposition}
\begin{proof}
Consider an abelian extension $ 0 \rightarrow M \rightarrow \Gamma \xrightarrow{\pi} G \rightarrow 1$. Suppose we have a collection of subgroups $\{\Gamma_{\scriptscriptstyle H}\}_{H \in \mathfrak{F}}$ of $\Gamma$, such that $\pi: \Gamma_{\scriptscriptstyle H} \xrightarrow{\cong} H$ for each $H \in \mathfrak{F}$. We will show that this collection of subgroups forms an $\mathfrak{F}$-structure. Let $x \in G$ and $H,K \in \mathfrak{F}$ and assume that $L:=x^{\scriptscriptstyle -1}H x\subseteq K$. It is easy to see that $\Gamma_{\scriptscriptstyle K}$ has a subgroup $P$ such that $\pi: P \xrightarrow{\cong} L$.

Now, for any $z \in \Gamma$ such that $\pi(z)=x$, we have $\pi(z^{\scriptscriptstyle -1}\Gamma_{\scriptscriptstyle H} z) = \pi(P)$ and $\pi: z^{\scriptscriptstyle -1}\Gamma_{\scriptscriptstyle H} z \xrightarrow{\cong} L$. Because  $L \in \overline{\mathfrak{F}}$, we have $H^1(L,M)=0$. This implies that all splittings of the standard split extension $M\rtimes L$ are $M$-conjugate. Since any subgroup $Q$ of $\Gamma$ such that $\pi: Q \xrightarrow{\cong} L$, defines a splitting of the standard split extension, it follows that all such subgroups $Q$ are $M$-conjugate. Therefore, we may conclude that there exists an $m \in M$ such that $m^{\scriptscriptstyle -1}z^{\scriptscriptstyle -1}\Gamma_{\scriptscriptstyle H} zm=P$. Next, we let $y=zm$ and note that $\pi(y)=x$ and $y^{\scriptscriptstyle -1}\Gamma_{\scriptscriptstyle H}y = P \subseteq \Gamma_{\scriptscriptstyle K}$. This proves that $\{\Gamma_{\scriptscriptstyle H}\}_{H \in \mathfrak{F}}$ forms an $\mathfrak{F}$-structure.

To prove that $\mathrm{Str}_{\mathfrak{F}}(G,M)$ can be viewed is a subspace of $\mathrm{Ext}(G,M)$, we need to show that the map $\mathrm{Str}_{\mathfrak{F}}(G,M) \rightarrow \mathrm{Ext}(G,M)$, that maps an equivalence class of $\mathfrak{F}$-structures on $(G,M)$ to the underlying equivalence class of extensions of $G$ by $M$, is injective. This comes down to showing that $\mathfrak{F}$-structures on equivalent extensions of $G$ by $M$ are always equivalent. But this follows easily from the fact that all subgroups $Q$ of $\Gamma$, where $0 \rightarrow M \rightarrow \Gamma \xrightarrow{\pi} G$ is an abelian extension of $G$ by $M$ and $\pi: Q \xrightarrow{\cong} H \in \mathfrak{F}$, are $M$-conjugate.
The final statement is a straightforward application of the first statement of the proposition.
\end{proof}
Next, we show that an $\mathfrak{F}$-structure on an abelian extension of $G$ by $M$ gives rise to an abelian right extension of $\orb$ by $\underline{M}$.
\begin{proposition} \label{prop: fstruc to ext} Given an $\mathfrak{F}$-structure on an abelian extension of $G$ by $M$
\[ 0 \rightarrow M \rightarrow \Gamma \xrightarrow{p} G \rightarrow 1, \]
the induced functors $\pi: \mathcal{O}_{\mathfrak{F}}(\Gamma) \rightarrow \orb$ and $i: \mathfrak{M} \rightarrow \mathcal{O}_{\mathfrak{F}}(\Gamma)$, where $\mathfrak{M}$ is the category associated to $\{M^{\scriptscriptstyle H}\}_{H \in \mathfrak{F}}$, form an abelian right extension of $\orb$ by $\underline{M}$
\[  \mathfrak{M} \xrightarrow{i} \mathcal{O}_{\mathfrak{F}}(\Gamma) \xrightarrow{\pi} \orb,   \]
that is split if the given $\mathfrak{F}$-structure is split. Furthermore, equivalent $\mathfrak{F}$-structures on $(G,M)$ give rise to equivalent abelian right extensions of $\orb$ by $\underline{M}$.
\end{proposition}
\begin{proof}
Let us first consider the functor $i: \mathfrak{M} \rightarrow \mathcal{O}_{\mathfrak{F}}(\Gamma)$. Take $m \in M^{\scriptscriptstyle H}$, for some $H \in \mathfrak{F}$. Then $$\bar{m} \in \mathrm{Mor}(M^{\scriptscriptstyle H},M^{\scriptscriptstyle H}) \mbox{\; and \;}i(\bar{m}): \Gamma/\Gamma_{\scriptscriptstyle H} \rightarrow \Gamma/\Gamma_{\scriptscriptstyle H} : \Gamma_{\scriptscriptstyle H} \mapsto m\Gamma_{\scriptscriptstyle H}.$$ This is well-defined because $hmh^{\scriptscriptstyle -1}=\pi(h)\cdot m=m$ for each $h \in \Gamma_{\scriptscriptstyle H}$. To see that $i$ is injective on morphisms, suppose that $m\Gamma_{\scriptscriptstyle H} = n\Gamma_{\scriptscriptstyle H}$ for some $m,n \in M^{\scriptscriptstyle H}$ and some $H \in \mathfrak{F}$. This implies that $m^{\scriptscriptstyle -1}n \in \Gamma_{\scriptscriptstyle H} \cap M=\{e\}$, so $m=n$. This proves that $i$ is injective on morphisms.

Next, we consider the functor  $\pi: \mathcal{O}_{\mathfrak{F}}(\Gamma) \rightarrow \orb$. The functor $\pi$ maps a morphism: $$\Gamma/\Gamma_{\scriptscriptstyle H} \rightarrow \Gamma/\Gamma_{\scriptscriptstyle K} : \Gamma_{\scriptscriptstyle H} \mapsto x\Gamma_{\scriptscriptstyle K}$$ to the morphism: $$G/H \rightarrow G/K : H \mapsto \pi(x)K.$$ This is clearly well-defined. To see that $\pi$ is surjective on morphisms, consider a morphism $$G/H \rightarrow G/K : H \mapsto yK$$ in $\orb$. By the definition of $\mathfrak{F}$-structure, there is $x \in \Gamma$ such that $x^{\scriptscriptstyle -1}\Gamma_{\scriptscriptstyle H} x \subseteq \Gamma_{\scriptscriptstyle K}$ and $\pi(x)=y$. It follows that  $$\Gamma/\Gamma_{\scriptscriptstyle H} \rightarrow \Gamma/\Gamma_{\scriptscriptstyle K} : \Gamma_{\scriptscriptstyle H} \mapsto x\Gamma_{\scriptscriptstyle K}$$ is a morphism in $\mathcal{O}_{\mathfrak{F}}(\Gamma)$ that is mapped to  $$G/H \rightarrow G/K : H \mapsto yK$$ by $\pi$. Hence $\pi$ is surjective on morphisms. Since $p(M)=\{e\}$, it is  clear that $\pi \circ i = id$.

Now, let us check the final property of abelian right extensions.
Suppose we have two morphisms: $$\varphi: \Gamma/\Gamma_{\scriptscriptstyle H} \rightarrow \Gamma/\Gamma_{\scriptscriptstyle K} : \Gamma_{\scriptscriptstyle H} \mapsto x\Gamma_{\scriptscriptstyle K}$$ and $$\psi: \Gamma/\Gamma_{\scriptscriptstyle H} \rightarrow \Gamma/\Gamma_{\scriptscriptstyle K} : \Gamma_{\scriptscriptstyle H} \mapsto y\Gamma_{\scriptscriptstyle K}$$ in $\mathcal{O}_{\mathfrak{F}}(\Gamma)$ such that $\pi(\varphi)=\pi(\psi)$. Since $\pi(\Gamma_{\scriptscriptstyle K})=K$ and the group extension is exact,  this implies that there exists  $m \in M$ and  $k \in \Gamma_{\scriptscriptstyle K}$ such that $mx=yk$. For an arbitrary $h \in \Gamma_{\scriptscriptstyle H}$, we compute $(hmh^{\scriptscriptstyle -1})x = hmxx^{\scriptscriptstyle -1}h^{\scriptscriptstyle -1}x = hmxk_1$, for some $k_1 \in \Gamma_{\scriptscriptstyle K}$ since $x^{\scriptscriptstyle -1}\Gamma_{\scriptscriptstyle H}x \subseteq \Gamma_{\scriptscriptstyle K}$.  Using the fact that $mx=yk$, we obtain $(hmh^{\scriptscriptstyle -1})x=hyk_2$ for some $k_2 \in \Gamma_{\scriptscriptstyle K}$. Multiplying $hyk_2$ on the left with $yy^{\scriptscriptstyle -1}$ and using the fact that  $y^{\scriptscriptstyle -1}\Gamma_{\scriptscriptstyle H}y \subseteq \Gamma_{\scriptscriptstyle K}$, we find  $k_3 \in \Gamma_{\scriptscriptstyle K}$ such that $(hmh^{\scriptscriptstyle -1})x=yk_3$. It follows that $(hmh^{\scriptscriptstyle -1})x=mxk^{\scriptscriptstyle -1}k_3$. Multiplying both sides with $x^{\scriptscriptstyle -1}$ on the left and using normality of $M$ in $\Gamma$, it follows that $k^{\scriptscriptstyle -1}k_3 \in \Gamma_{\scriptscriptstyle K} \cap M=\{e\}$. This implies that $(hmh^{\scriptscriptstyle -1})x=mx$. Therefore, $hmh^{\scriptscriptstyle -1}=m$ for any $h \in \Gamma_{\scriptscriptstyle H}$. It follows that $m \in M^{\scriptscriptstyle H}$, since $\pi(\Gamma_{\scriptscriptstyle H})=H$ and $m=hmh^{\scriptscriptstyle -1}=\pi(h)\cdot m$. It now follows easily from $mx=yk$, that $\varphi \circ i(\bar{m})=\psi$.

It still remains to show that $m$ is the unique element in $M^{\scriptscriptstyle H}$ with this property. Suppose that $\varphi \circ i(\bar{n})=\psi$ for some $n \in M^{\scriptscriptstyle H}$. Then $\varphi \circ i(\bar{m})=\varphi \circ i(\bar{n})$. Hence, there exists a $k \in \Gamma_{\scriptscriptstyle K}$ such that $mx=nxk$.  Multiplying on the left with $x^{\scriptscriptstyle -1}$ and using normality of $M$ in $\Gamma$, we see that $k \in \Gamma_{\scriptscriptstyle K} \cap M=\{e\}$. Hence, $mx=nx$  and so $m=n$.

Next, let us prove that the $\orb$-module structure on $\underline{M}$ induced by the extension coincides with the given one. Consider an arbitrary  morphism: $$\varphi: G/H \rightarrow G/K : H \mapsto yK$$ and  $m \in M^{\scriptscriptstyle K}$. Let $$\theta: \Gamma/\Gamma_{\scriptscriptstyle H} \rightarrow \Gamma/\Gamma_{\scriptscriptstyle K} : \Gamma_{\scriptscriptstyle H} \mapsto x\Gamma_{\scriptscriptstyle K}$$ be a morphism in $\mathcal{O}_{\mathfrak{F}}(\Gamma)$ such that $\pi(\theta)=\varphi$. Since $\pi(x)K=yK$ and $m \in M^{\scriptscriptstyle K}$, it follows that  $y\cdot m = xmx^{\scriptscriptstyle -1}$. Therefore, $i(\bar{m}) \circ \theta = \theta \circ i(\overline{y\cdot m})$, hence $\underline{M}(\varphi)(m)=y\cdot m$. This shows that $\mathfrak{M} \xrightarrow{i} \mathcal{O}_{\mathfrak{F}}(\Gamma) \xrightarrow{\pi} \orb$ is an abelian right extension of $\orb$ by $\underline{M}$.

Assume that the given $\mathfrak{F}$-structure splits, via the splitting $s: G \rightarrow \Gamma$. Let $(m_{\scriptscriptstyle H})_{H \in \mathfrak{F}}$ be the elements of $M$ such that $s(H)=m_{\scriptscriptstyle H}^{\scriptscriptstyle -1}\Gamma_{\scriptscriptstyle H}m_{\scriptscriptstyle H}$ for each $H \in \mathfrak{F}$. One can verify that the functor $S: \orb \rightarrow  \mathcal{O}_{\mathfrak{F}}(\Gamma)$, that takes $G/H$ to $\Gamma/\Gamma_{\scriptscriptstyle H}$ and maps a morphism: $$\varphi: G/H \rightarrow G/K : H \mapsto xK$$ to the morphism: $$\Gamma/\Gamma_{\scriptscriptstyle H} \rightarrow \Gamma/\Gamma_{\scriptscriptstyle K} : \Gamma_{\scriptscriptstyle H} \mapsto m_{\scriptscriptstyle H}s(x)m_{\scriptscriptstyle K}^{\scriptscriptstyle -1}\Gamma_{\scriptscriptstyle K},$$ is a well-defined splitting functor for $\mathfrak{M} \xrightarrow{i} \mathcal{O}_{\mathfrak{F}}(\Gamma) \xrightarrow{\pi} \orb$. \\
\indent Finally, suppose we have two $\mathfrak{F}$-structures $\{\Gamma_{1_{\scriptscriptstyle H}}\}_{H \in \mathfrak{F}}$ and $\{\Gamma_{2_{\scriptscriptstyle H}}\}_{H \in \mathfrak{F}}$  on the extensions: $$M \xrightarrow{i_1} \Gamma_1 \xrightarrow{\pi_1} G \mbox{\; and \;} M \xrightarrow{i_2} \Gamma_2 \xrightarrow{\pi_2} G,$$ respectively. Suppose also that they are equivalent via a group isomorphism $\theta : \Gamma_1 \rightarrow \Gamma_2$. Let $(m_{\scriptscriptstyle H})_{H \in \mathfrak{F}}$ be the elements of $M$ such that $\theta(\Gamma_{1_{\scriptscriptstyle H}})=m_{\scriptscriptstyle H}^{\scriptscriptstyle -1}\Gamma_{2_{\scriptscriptstyle H}}m_{\scriptscriptstyle H}$ for all $H \in \mathfrak{F}$. Now, we define $\Theta : \mathcal{O}_{\mathfrak{F}}(\Gamma_1) \rightarrow \mathcal{O}_{\mathfrak{F}}(\Gamma_2)$ by $\Theta(\Gamma_1/\Gamma_{1_{\scriptscriptstyle H}})=\Gamma_2/\Gamma_{2_{\scriptscriptstyle H}}$ and $$\Theta(\Gamma_1/\Gamma_{1_{\scriptscriptstyle H}} \rightarrow \Gamma_1/\Gamma_{1_{\scriptscriptstyle K}}: \Gamma_{1_{\scriptscriptstyle H}} \mapsto x\Gamma_{1_{\scriptscriptstyle K}})=\Gamma_2/\Gamma_{2_{\scriptscriptstyle H}} \rightarrow \Gamma_2/\Gamma_{2_{\scriptscriptstyle K}}: \Gamma_{2_{\scriptscriptstyle H}} \mapsto m_{\scriptscriptstyle H}\theta(x)m_{\scriptscriptstyle K}^{\scriptscriptstyle -1}\Gamma_{2_{\scriptscriptstyle K}}.$$ It is not difficult to checks $\Theta$ is a functor that entails an equivalence of extensions between $\mathfrak{M} \xrightarrow{i_1} \mathcal{O}_{\mathfrak{F}}(\Gamma_1) \xrightarrow{\pi_1} \orb$ and $\mathfrak{M} \xrightarrow{i_2} \mathcal{O}_{\mathfrak{F}}(\Gamma_2) \xrightarrow{\pi_2} \orb$. This proves the proposition.
\end{proof}
By the proposition, we have a well-defined map:
\[ \Psi: \mathrm{Str}_{\mathfrak{F}}(G,M) \rightarrow \mathrm{Ext}_{\mathfrak{F}}(G,\underline{M}). \]
In the remainder of this section, our goal is to prove that $\Psi$ is a bijection. First, we need a few lemmas.

Consider an abelian right extension of $\orb$ by $\underline{M}$:
\[ \mathfrak{M} \xrightarrow{i} \mathfrak{E} \xrightarrow{\pi} \orb. \]
\begin{lemma}\label{lemma: key lemma1} The group $\Gamma=\mathrm{Mor}(E(\{e\}),E(\{e\}))$ fits into an abelian extension:
\[0 \rightarrow M \xrightarrow{i} \Gamma \xrightarrow{\pi} G \rightarrow 1.\]
\end{lemma}
\begin{proof}
Obviously, $\Gamma$ surjects onto $\mathrm{Mor}(G/\{e\},G/\{e\})$ via $\pi$ and $\mathrm{Mor}(\underline{M}(G/\{e\}),\underline{M}(G/\{e\}))$ injects into $\Gamma$ via $i$, by the definition of extensions. It is clear that $\mathrm{Mor}(G/\{e\},G/\{e\})$, with group law defined by $\varphi\cdot\psi = \psi \circ \varphi$ is isomorphic to $G$. Similarly,  $\mathrm{Mor}(\underline{M}(G/\{e\}),\underline{M}(G/\{e\}))$ is an abelian group isomorphic to $M$. From now on we will identify the group $\mathrm{Mor}(G/\{e\},G/\{e\})$ with $G$ and $\mathrm{Mor}(\underline{M}(G/\{e\}),\underline{M}(G/\{e\}))$ with $M$. Using the properties of abelian extensions of the orbit category, it is not difficult to show that $\Gamma$ is a group, with group law defined by $\bar{\varphi}\cdot\bar{\psi}= \bar{\psi} \circ \bar{\varphi}$ such that $0 \rightarrow M \xrightarrow{i} \Gamma \rightarrow G \rightarrow 1$ is a short exact sequence of groups that induces a $G$-module structure on $M$ which coincides with the original one.
\end{proof}
\begin{notation}\rm ${\;}$
\begin{enumerate}[(a)]
\item From now on, we identify $\Gamma$ with $\mathrm{Mor}(E(\{e\}),E(\{e\}))$, $\mathrm{Mor}(G/\{e\},G/\{e\})$ with $G$ and $\mathrm{Mor}(\underline{M}(G/\{e\}),\underline{M}(G/\{e\}))$ with $M$. Therefore, an element $x\in G$ will sometimes be viewed as a morphism $x: G/\{e\} \xrightarrow{x} G/\{e\}$ or vice versa. We use similar notation for elements of $M$ and $\Gamma$.
\medskip
\item To keep the notation as light as possible, we use $\pi$ to denote morphisms on the orbit category extension and on the group extension. The distinction will be clear from the context. When we are working on the group extension, we will drop the map $i$ from the notation.
\end{enumerate}
\end{notation}
For each $H \in \mathfrak{F}$, fix a morphism $\varphi_{\scriptscriptstyle H}\in\mathrm{Mor}(E(\{e\}),E(H))$ such that $$\pi(\varphi_{\scriptscriptstyle H})=(id_{\scriptscriptstyle H}:G/\{e\} \rightarrow G/H: \{e\} \mapsto H).$$
Since the map $\underline{M}(id_{\scriptscriptstyle H})$ is just the inclusion $M^{\scriptscriptstyle H} \rightarrow M$ and $\mathfrak{M} \xrightarrow{i} \mathfrak{E} \xrightarrow{\pi} \orb$ is an extension that induces the original $\orb$-module structure on $\underline{M}$, we conclude that $\varphi_{\scriptscriptstyle H} \circ i(\bar{m})=i(\bar{m}) \circ \varphi_{\scriptscriptstyle H}$ for all $m \in M^{\scriptscriptstyle H}$ and all $H \in \mathfrak{F}$.
\begin{definition} \label{def: action} \rm For each $H \in \mathfrak{F}$, define the left group action:
\[ \alpha_{\scriptscriptstyle H} : \Gamma \times \mathrm{Mor}(E(\{e\}),E(H)) \rightarrow \mathrm{Mor}(E(\{e\}),E(H)) : (\theta,\varphi) \mapsto \varphi \circ \theta. \]
Let $\Gamma_{\scriptscriptstyle H}$ be the stabilizer of $\varphi_{\scriptscriptstyle H}$ under this action. Let $\mathfrak{F}$ be the family of subgroups of $\Gamma$ that contains all the stabilizers $\Gamma_{\scriptscriptstyle H}$.
\end{definition}
We observe that $\Gamma_{\{e\}}=\{id\}$, all the elements of $M^{\scriptscriptstyle H}$ commute with the elements of $\Gamma_{\scriptscriptstyle H}$ and  $\Gamma_{\scriptscriptstyle H} \cap M=\{e\}$.
\begin{lemma} \label{lemma: key lemma2} The collection of subgroups $\{\Gamma_{\scriptscriptstyle H}\}_{H \in \mathfrak{F}}$ forms an $\mathfrak{F}$-structure on:
\[0 \rightarrow M \rightarrow \Gamma \xrightarrow{\pi} G \rightarrow 1.\]
\end{lemma}
\begin{proof}
Let us first check property (i) of the definition of $\mathfrak{F}$-structure. Fix  $H \in \mathfrak{F}$ and  $h \in H$. For the morphism: $${h}: G/\{e\} \rightarrow G/\{e\} : \{e\} \mapsto h\{e\},$$ we can find  $\omega \in \Gamma$ such that $\pi(\omega)={h}$. Since $\pi(\varphi_{\scriptscriptstyle H} \circ \omega)=id_{\scriptscriptstyle H}=\pi(\varphi_{\scriptscriptstyle H})$, we can find  $m \in M$ such that $\varphi_{\scriptscriptstyle H} \circ \omega \circ i(\bar{m})=\varphi_{\scriptscriptstyle H}$. Hence, for $\theta=\omega\circ i(\bar{m})$, we have  $\theta \in \Gamma_{\scriptscriptstyle H}$ such that $\pi(\theta)={h}$. We conclude that $\pi(\Gamma_{\scriptscriptstyle H})=H$. Since the kernel of $\pi:\Gamma \rightarrow G$ is $M$ and $\Gamma_{\scriptscriptstyle H} \cap M = \{e\}$, it follows that $\pi:\Gamma_{\scriptscriptstyle H} \xrightarrow{\cong} H$. Hence, property (i) is satisfied.

To check property (ii), let $x \in G$ and $H,K \in \mathfrak{F}$ and assume that $x^{\scriptscriptstyle -1}H x\subseteq K$. This means that we have a morphism: $$\varphi: G/H \rightarrow G/K : H \mapsto xK \in \orb.$$ By the surjectivity of $\pi: \mathfrak{E} \rightarrow \orb$ on morphisms, we find a $\psi \in \mathrm{Mor}(E(H),E(K))$ such that $\pi(\psi)=\varphi$. It follows that there exists $\omega \in \Gamma$ such that $\pi(\omega)=x$ and  $\pi(\varphi_{\scriptscriptstyle K} \circ \omega)=\pi(\psi \circ \varphi_{\scriptscriptstyle H})$. Hence, there is  $m \in M$ such that $\psi \circ \varphi_{\scriptscriptstyle H}= \varphi_{\scriptscriptstyle K} \circ \omega \circ i(\bar{m})$. Setting $\theta=\omega \circ i(\bar{m})$, we find $\theta \in \Gamma$ such that $\pi(\theta)=x$ and $\psi \circ \varphi_{\scriptscriptstyle H}= \varphi_{\scriptscriptstyle K} \circ \theta$. We claim that $\theta^{\scriptscriptstyle -1}\Gamma_{\scriptscriptstyle H}\theta \subseteq \Gamma_{\scriptscriptstyle K}$. Indeed, for an arbitrary element $h \in \Gamma_{\scriptscriptstyle H}$, we compute:
\begin{eqnarray*}
\varphi_{\scriptscriptstyle K} \circ (\theta^{\scriptscriptstyle -1}h\theta)&=& \varphi_{\scriptscriptstyle K} \circ \theta \circ h \circ \theta^{\scriptscriptstyle -1} \\
&=& \psi \circ \varphi_{\scriptscriptstyle H} \circ h \circ \theta^{\scriptscriptstyle -1} \\
&=& \psi \circ \varphi_{\scriptscriptstyle H} \circ \theta^{\scriptscriptstyle -1} \\
&=& \varphi_{\scriptscriptstyle K} \circ \theta \circ \theta^{\scriptscriptstyle -1} \\
&=& \varphi_{\scriptscriptstyle K},
\end{eqnarray*}
which proves the claim. This proves that $\{\Gamma_{\scriptscriptstyle H}\}_{H \in \mathfrak{F}}$ forms an $\mathfrak{F}$-structure.
\end{proof}
The following lemma reveals an important property of $\mathfrak{E}$.
\begin{lemma}\label{lemma: key lemma} Let $\psi_1,\psi_2 \in \mathrm{Mor}(E(H),E(K))$ for some $H,K \in \mathfrak{F}$. If $\psi_1 \circ \varphi_{\scriptscriptstyle H} = \psi_2 \circ \varphi_{\scriptscriptstyle H}$, then $\psi_1=\psi_2$.
\end{lemma}
\begin{proof}
Suppose $\psi_1 \circ \varphi_{\scriptscriptstyle H} = \psi_2 \circ \varphi_{\scriptscriptstyle H}$. Applying $\pi$, we obtain $\pi(\psi_1) \circ id_{\scriptscriptstyle H} = \pi(\psi_2) \circ id_{\scriptscriptstyle H}$. In an orbit category, this implies that $\pi(\psi_1)=\pi(\psi_2)$. Hence, we find an $m \in M^{\scriptscriptstyle H}$ such that $\psi_2 = \psi_1 \circ i(\bar{m})$. We now compute
$\psi_1 \circ \varphi_{\scriptscriptstyle H}  =  \psi_2 \circ \varphi_{\scriptscriptstyle H}
 =  \psi_1 \circ i(\bar{m}) \circ \varphi_{\scriptscriptstyle H}
 =  \psi_1 \circ \varphi_{\scriptscriptstyle H} \circ i(\bar{m})$.
So, uniqueness yields $i(\bar{m})=id$. We conclude that $\psi_1=\psi_2$.
\end{proof}
We can now prove the main theorem of this section.
\begin{theorem} \label{th: special ext}
Let $G$ be a group, let $\mathfrak{F}$ be a family of subgroups of $G$ that contains the trivial subgroup and suppose $M \in  G\mbox{-mod}$. There is a one-to-one correspondence between $\mathrm{Ext}_{\mathfrak{F}}(G,\underline{M})$ and $\mathrm{Str}_{\mathfrak{F}}(G,M)$, such that the class of split extension of $\orb$ by $\underline{M}$ corresponds to the class of split $\mathfrak{F}$-structures on $(G,M)$.
\end{theorem}

\begin{proof}To prove the theorem it suffices to show that the map:
\[ \Psi: \mathrm{Str}_{\mathfrak{F}}(G,M) \rightarrow \mathrm{Ext}_{\mathfrak{F}}(G,\underline{M}), \]
constructed earlier, is a bijection. Let us first prove that $\Psi$ is injective. \\
Suppose we have  $\mathfrak{F}$-structures $\{\Gamma_{1_{\scriptscriptstyle H}}\}_{H \in \mathfrak{F}}$ and  $\{\Gamma_{2_{\scriptscriptstyle H}}\}_{H \in \mathfrak{F}}$ on extensions: $$M \rightarrow \Gamma_1 \xrightarrow{\pi_1} G \mbox{\; and \;} M \rightarrow \Gamma_2 \xrightarrow{\pi_2} G,$$ respectively, such that the induced abelian right extensions: $$\mathfrak{M} \xrightarrow{i_1} \mathcal{O}_{\mathfrak{F}}(\Gamma_1) \xrightarrow{\pi_1} \orb \mbox{\; and \;} \mathfrak{M} \xrightarrow{i_2} \mathcal{O}_{\mathfrak{F}}(\Gamma_2) \xrightarrow{\pi_2} \orb$$ are equivalent via a functor $\Theta$.
Viewing $\Theta$ as a map from $\Gamma_1=\mathrm{Mor}(\Gamma_1/\{e\},\Gamma_1/\{e\})$ to $\Gamma_2=\mathrm{Mor}(\Gamma_2/\{e\},\Gamma_2/\{e\})$, it becomes a group homomorphism $\Theta:\Gamma_1 \rightarrow \Gamma_2$ such that the diagram:
\[ \xymatrix{  0 \ar[r] & M \ar[r] \ar[d]^{id} & \Gamma_1 \ar[r]^{\pi_1} \ar[d]^{\Theta} & G \ar[d]^{id}  \ar[r]& 1   \\
 0 \ar[r] & M \ar[r]  & \Gamma_2 \ar[r]^{\pi_2}  & G \ar[r] & 1 } \]
commutes. 
Denote $\Gamma_i/\{e\} \rightarrow \Gamma_i/\Gamma_{i_{\scriptscriptstyle H}} : \{e\} \mapsto \Gamma_{i_{\scriptscriptstyle H}}$ by $\varphi_{i_{\scriptscriptstyle H}}$, for $i=1,2$. 
We observe that for each $H \in \mathfrak{F}$, the functor $\Theta: \mathcal{O}_{\mathfrak{F}}(\Gamma_1) \rightarrow  \mathcal{O}_{\mathfrak{F}}(\Gamma_2)$ maps $\varphi_{{1_{\scriptscriptstyle H}}}$ to $\varphi_{2_{\scriptscriptstyle H}} \circ i_2( \bar{m}_{\scriptscriptstyle H})$, for some $m_{\scriptscriptstyle H} \in M$. 
Now consider the action from Definition \ref{def: action}, for $i=1,2$, and note that the stabilizer group of $\varphi_{i_{\scriptscriptstyle H}}$ is $\Gamma_{i_{\scriptscriptstyle H}}$. One easily checks that $\Theta$ maps the  stabilizer  of $\varphi_{1_{\scriptscriptstyle H}}$  onto the stabilizer of $\varphi_{2_{\scriptscriptstyle H}}\circ i_2( \bar{m}_{\scriptscriptstyle H})$ which is $m_{\scriptscriptstyle H}^{\scriptscriptstyle -1}\Gamma_{2_{\scriptscriptstyle H}}m_{\scriptscriptstyle H}$. We conclude that $\Theta(\Gamma_{1_{\scriptscriptstyle H}})=m_{\scriptscriptstyle H}^{\scriptscriptstyle -1}\Gamma_{2_{\scriptscriptstyle H}}m_{\scriptscriptstyle H}$ for all $H \in \mathfrak{F}$. Hence, the given two $\mathfrak{F}$-structures are equivalent. This proves that $\Psi$ is injective. \\
To prove that $\Psi$ is surjective, choose an abelian right extension of $\orb$ by $\underline{M}$ denoted:
\begin{equation} \label{eq: ext} \mathfrak{M} \xrightarrow{i} \mathfrak{E} \xrightarrow{\pi} \orb. \end{equation}
In Lemmas \ref{lemma: key lemma1} and \ref{lemma: key lemma2}, we saw that this entails an $\mathfrak{F}$-structure $\{\Gamma_{\scriptscriptstyle H}\}_{H \in \mathfrak{F}}$ on
\[ 0 \rightarrow M \rightarrow \Gamma \xrightarrow{\pi} G \rightarrow 1. \]
By definition, $\Psi$ maps the equivalence class of this $\mathfrak{F}$-structure to the equivalence class of the extension:
\[  \mathfrak{M} \xrightarrow{i} \mathcal{O}_{\mathfrak{F}}(\Gamma) \xrightarrow{\pi} \orb.   \]
We prove that this extension is equivalent to (\ref{eq: ext}).

Consider a morphism: $$\Gamma/\Gamma_{\scriptscriptstyle H} \rightarrow \Gamma/\Gamma_{\scriptscriptstyle K} : \Gamma_{\scriptscriptstyle H} \mapsto \theta \Gamma_{\scriptscriptstyle K}.$$
So, we have $\theta^{\scriptscriptstyle -1}\Gamma_{\scriptscriptstyle H} \theta \subseteq \Gamma_{\scriptscriptstyle K}$. Applying $\pi$ to this equation, we obtain $\pi(\theta)^{\scriptscriptstyle -1}H \pi(\theta)\subseteq K$. Hence, we have a well-defined morphism: $$\varphi: G/H \rightarrow G/K : H \mapsto \pi(\theta)K.$$
By surjectivity, we find  $\phi \in \mathrm{Mor}(E(H),E(K))$ such that $\pi(\phi)=\varphi$. This implies that $\pi(\phi \circ \varphi_{\scriptscriptstyle H})=\pi(\varphi_{\scriptscriptstyle K} \circ \theta)$. Hence, we find a unique $m \in M$ such that $\phi \circ \varphi_{\scriptscriptstyle H} \circ i(\bar{m})=\varphi_{\scriptscriptstyle K} \circ \theta$. We now claim that $m \in M^{\scriptscriptstyle H}$.
To prove this, let $h \in H$. Because $\Gamma_{\scriptscriptstyle H}$ surjects onto $H$, there exists $\alpha \in \Gamma_{\scriptscriptstyle H}$ such that $\pi(\alpha)=h$. We need to show that $h\cdot m=m$. Recall that $i(\bar{h \cdot m})=\alpha^{\scriptscriptstyle -1} \circ i(\bar{m}) \circ \alpha$. Hence, by uniqueness, it suffices to show that $ \phi \circ \varphi_{\scriptscriptstyle H} \circ \alpha^{\scriptscriptstyle -1} \circ i(\bar{m}) \circ \alpha=\varphi_{\scriptscriptstyle K} \circ \theta$. Using the fact  $\alpha^{\scriptscriptstyle -1} \in \Gamma_{\scriptscriptstyle H}$,  we compute $ \phi \circ \varphi_{\scriptscriptstyle H} \circ \alpha^{\scriptscriptstyle -1} \circ i(\bar{m}) \circ \alpha = \phi \circ \varphi_{\scriptscriptstyle H} \circ i(\bar{m}) \circ \alpha = \varphi_{\scriptscriptstyle K} \circ \theta \circ \alpha$. Because $\theta^{\scriptscriptstyle -1}\Gamma_{\scriptscriptstyle H} \theta \subseteq \Gamma_{\scriptscriptstyle K}$, there exists $\beta \in \Gamma_{\scriptscriptstyle K}$ such that $\beta \circ \theta = \theta \circ \alpha$. It follows that $ \phi \circ \varphi_{\scriptscriptstyle H} \circ \alpha^{\scriptscriptstyle -1} \circ i(\bar{m}) \circ \alpha = \varphi_{\scriptscriptstyle K} \circ \beta \circ \theta= \varphi_{\scriptscriptstyle K} \circ \theta$.  This proves that $m \in M^{\scriptscriptstyle H}$. Therefore, we have that $ \phi \circ i(\bar{m}) \circ \varphi_{\scriptscriptstyle H}  = \phi \circ \varphi_{\scriptscriptstyle H} \circ i(\bar{m})  = \varphi_{\scriptscriptstyle K} \circ \theta$.
Setting $\psi=\phi \circ i(\bar{m})$, we obtain $\psi \in \mathrm{Mor}(E(H),E(K))$ such that $\psi \circ \varphi_{\scriptscriptstyle H}=\varphi_{\scriptscriptstyle K} \circ \theta$. By Lemma \ref{lemma: key lemma}, $\psi$ is the unique morphism with this property. We claim that the functor:
\[ \Lambda : \orbg \rightarrow \mathfrak{E}, \]
with $\Lambda(\Gamma/\Gamma_{\scriptscriptstyle H})=E(H)$ and $\Lambda(\Gamma/\Gamma_{\scriptscriptstyle H} \rightarrow \Gamma/\Gamma_{\scriptscriptstyle K} : \Gamma_{\scriptscriptstyle H} \mapsto \theta \Gamma_{\scriptscriptstyle K})=\psi$, where $\psi$ is the unique element in $\mathrm{Mor}(E(H),E(K))$ such that $\psi \circ \varphi_{\scriptscriptstyle H}=\varphi_{\scriptscriptstyle K} \circ \theta$, is a functor that will provide the desired equivalence.

To prove this, first note that the map $\Lambda$ is well defined. Indeed, if we take $\omega=\theta k=k \circ \theta$ instead of $\theta$ as a representative of $\theta\Gamma_{\scriptscriptstyle K}$ (here $k \in \Gamma_{\scriptscriptstyle K}$), then $\psi$ is still the unique morphism such that $\psi \circ \varphi_{\scriptscriptstyle H}= \varphi_{\scriptscriptstyle K} \circ \omega$ (because $\varphi_{\scriptscriptstyle K} \circ k  = \varphi_{\scriptscriptstyle K}$).

Now, consider the identity morphism $id \in \mathrm{Mor}(\Gamma/\Gamma_{\scriptscriptstyle H},\Gamma/\Gamma_{\scriptscriptstyle H})$ for some $\Gamma_{\scriptscriptstyle H} \in \mathfrak{F}$. Then $id \in \mathrm{Mor}(E(H),E(H))$ is clearly the unique morphism such that $id \circ \varphi_{\scriptscriptstyle H} = \varphi_{\scriptscriptstyle H} \circ id$. Hence, $\Lambda(id)=id$.

Finally, let $$f_1 : \Gamma/\Gamma_{\scriptscriptstyle H} \rightarrow \Gamma/\Gamma_{\scriptscriptstyle K} : \Gamma_{\scriptscriptstyle H} \mapsto \theta_1 \Gamma_{\scriptscriptstyle K} \mbox{\; and \;} f_2: \Gamma/\Gamma_{\scriptscriptstyle K} \rightarrow \Gamma/\Gamma_L : \Gamma_{\scriptscriptstyle K} \mapsto \theta_2 \Gamma_L$$ and suppose that $\psi_1 \circ \varphi_{\scriptscriptstyle H}=\varphi_{\scriptscriptstyle K} \circ \theta_1$ and $\psi_2 \circ \varphi_{\scriptscriptstyle K}=\varphi_L \circ \theta_2$. Then, clearly, it follows: $$f_2 \circ f_1 : \Gamma/\Gamma_{\scriptscriptstyle H} \rightarrow \Gamma/\Gamma_L : \Gamma_{\scriptscriptstyle H} \mapsto \theta_1\theta_2\Gamma_L$$ and we compute:
\begin{eqnarray*}
(\psi_2 \circ \psi_1) \circ \varphi_{\scriptscriptstyle H} & = & \psi_2 \circ \varphi_{\scriptscriptstyle K} \circ \theta_1 \\
& = & \varphi_L \circ \theta_2 \circ \theta_1 \\
& = & \varphi_L \circ \theta_1\theta_2.
\end{eqnarray*}
We conclude that $\Lambda(f_2 \circ f_1)=\psi_2 \circ \psi_1=\Lambda(f_2)\circ \Lambda(f_1)$. Hence, $\Lambda$ is a functor. It is now an easy exercise to finish the proof of the claim and the theorem, by showing that
\[ \xymatrix{  \mathfrak{M} \ar[r]^{i} \ar[d]^{id} & \orbg \ar[r]^{\pi} \ar[d]^{\Lambda} & \orb \ar[d]^{id}     \\
  \mathfrak{M} \ar[r]^{i}  & \mathfrak{E} \ar[r]^{\pi}  & \orb   } \]
commutes.
\end{proof}

We have shown that, up to equivalence, abelian right extensions of the orbit category $\orb$ by a fixed point functor $\underline{M}$ are always of the form $\mathfrak{M} \rightarrow \orbg \xrightarrow{\pi} \orb$ and induced by abelian group extensions $0 \rightarrow M \rightarrow \Gamma \xrightarrow{\pi} G$.
It turns out that, in some sense, the converse is also true.
\begin{proposition}  Let $G$ be a group and let $\mathfrak{F}$ be a family of subgroups that contains the trivial subgroup. Suppose $M \in \orbmod$ and consider a right abelian extension of $\orb$ by $M$, $\mathfrak{M} \xrightarrow{i} \mathfrak{E} \xrightarrow{\pi} \orb, $
where $\mathfrak{M}$ is the category associated to $\{M(G/H)_{H \in \mathfrak{F}}\}$.
If $\mathfrak{E}$ is an orbit category $\mathcal{O}_{\mathfrak{H}}(\Gamma)$  with $\{e\} \in \mathfrak{H}$ such that $ \mathcal{O}_{\mathfrak{H}}(\Gamma) \xrightarrow{\pi} \orb$ is induced by a group homomorphism $\pi: \Gamma \rightarrow G$ and $\mathfrak{F}=\{ \pi(\Gamma_{\scriptscriptstyle H}) \ | \ \Gamma_{\scriptscriptstyle H} \in \mathfrak{H} \}$,
then $M$ is isomorphic to the fixed point functor $\underline{N}$, with $N=M(G/\{e\})$.
\end{proposition}
\begin{proof}
The group $M(G/\{e\})$ becomes a left $G$-module under the identification of $G$ with the group $\mathrm{Mor}(G/\{e\},G/\{e\})$.
Let $id_{\scriptscriptstyle H}$ be the morphism $G/e \rightarrow G/H : e \mapsto H$. We claim that for each $H \in \mathfrak{F}$, $$M(id_{\scriptscriptstyle H}): M(G/H) \rightarrow M(G/\{e\})$$ is an injective map that surjects onto $M(G/\{e\})^{\scriptscriptstyle H}$. \\ \indent  Indeed, take $m \in M(G/H)$ and suppose that $M(id_{\scriptscriptstyle H})(m)=0$. It follows that $i(\bar{m}) \circ id_{\Gamma_{\scriptscriptstyle H}}= id_{\Gamma_{\scriptscriptstyle H}}$ in $\mathcal{O}_{\mathfrak{H}}(\Gamma)$. Since this implies that $i(\bar{m})$ is the identity map on $\Gamma/\Gamma_{\scriptscriptstyle H}$, we conclude by injectivity that $m=0$. This shows that $M(id_{\scriptscriptstyle H})$ is injective.

 Next, we let $\theta \in H$ and view it as an element of $\mathrm{Mor}(G/\{e\},G/\{e\})$. Because $id_{\scriptscriptstyle H} \circ \theta = id_{\scriptscriptstyle H}$, it follows that $M(\theta)\circ M(id_{\scriptscriptstyle H})=M(id_{\scriptscriptstyle H})$. This implies  $M(id_{\scriptscriptstyle H})(M(G/H)) \subseteq M(G/\{e\})^{\scriptscriptstyle H}$.

 Now, let $m \in  M(G/\{e\})^{\scriptscriptstyle H}$. Using the fact that $\pi(\Gamma_{\scriptscriptstyle H})=H$ one can check that for any $\psi \in \Gamma_{\scriptscriptstyle H}$ considered as a morphism in $\mathrm{Mor}(\Gamma/\{e\},\Gamma/\{e\})$, we have $i(\bar{m}) \circ \psi = \psi \circ i(\bar{m})$. Let us consider the morphism: $$i(\bar{m}): \Gamma/e \rightarrow \Gamma/\{e\} : \{e\} \mapsto x\{e\}.$$ Because $i(\bar{m}) \circ \psi = \psi \circ i(\bar{m})$ for any $\psi \in \Gamma_{\scriptscriptstyle H} \subseteq \mathrm{Mor}(\Gamma/\{e\},\Gamma/\{e\})$,  we have a well-defined morphism: $$\Gamma/\Gamma_{\scriptscriptstyle H} \rightarrow \Gamma/\Gamma_{\scriptscriptstyle H} : \Gamma_{\scriptscriptstyle H} \mapsto x\Gamma_{\scriptscriptstyle H}.$$ Since $\pi(x)=e$, this morphism equals $i(\bar{n})$ for some $n \in M(G/H)$. It can now be readily verified that $i(\bar{n}) \circ id_{\Gamma_{\scriptscriptstyle H}} = id_{\Gamma_{\scriptscriptstyle H}} \circ i(\bar{m})$; thus, $M(id_{\scriptscriptstyle H})(n)=m$. This shows that $M(id_{\scriptscriptstyle H})(M(G/H)) = M(G/\{e\})^{\scriptscriptstyle H}$, and the claim follows.

Finally, let $N$ be the $G$-module $M(G/\{e\})$.
It follows that the isomorphisms: $$\eta(G/H): M(G/H) \rightarrow N^{\scriptscriptstyle H}: m \mapsto M(id_{\scriptscriptstyle H})(m)$$ assemble to form a natural equivalence of  functors $\eta : M \rightarrow \underline{N}$.
\end{proof}
\section{A standard cochain complex} \label{sec: complex}
\indent In this section, we define a cochain complex that will be used to compute Bredon cohomology. Let us first recall a standard free resolution of $\underline{\mathbb{Z}}$  (see \cite{Waner}). Throughout, let $G$ be a group, let $\mathfrak{F}$ a family of subgroups of $G$ and let $M$ be a $\orb$-module. For each $n \in \mathbb{N}$, we define the $\mathbb{Z}$-module:
\medskip
\[ \mathbb{Z}[G/H_0,\ldots,G/H_n]= \mathbb{Z}[G/H_0,G/H_1]\otimes \ldots \otimes \mathbb{Z}[G/H_{n-1},G/H_n]. \]

\medskip

Viewing $\mathbb{Z}[G/H_0,\ldots,G/H_n]$ as a trivial $G$-module, we define the $\orb$-module:
\[ B_n(\mathfrak{F},G)= \bigoplus_{\scriptstyle (H_0,\ldots,H_n)\in \mathfrak{F}^{n+1}} \mathbb{Z}[-,G/H_0]\otimes \underline{\mathbb{Z}[G/H_0,\ldots,G/H_n]}\]
for each $n \in \mathbb{N}$. For each $n \in \mathbb{N}_0$, we also define natural transformations $d_n: B_n(\mathfrak{F},G) \rightarrow B_{n-1}(\mathfrak{F},G)$ as follows: $\forall H \in \mathfrak{F}$: \[ d_n(G/H): B_n(\mathfrak{F},G)(G/H) \rightarrow B_{n-1}(\mathfrak{F},G)(G/H): x \mapsto \sum_{i=0}^n(-1)^i\delta_i(x),\] where $\delta_i$ are defined on the generators of  $B_n(\mathfrak{F},G)(G/H)$ by: $\delta_i(f\otimes f_1 \otimes \ldots \otimes f_n) =$
\medskip
{$$=\left\{\begin{array}{lll}
 (f_1 \circ f) \otimes f_2 \otimes \ldots f_n  &\in   \mathbb{Z}[G/H,G/H_1]\otimes \mathbb{Z}[G/H_1,\ldots,G/H_n]\\
  &\mbox{\;\;\; if \;}  i=0;\\
                                                          f \otimes f_1 \otimes \ldots \otimes (f_{i+1} \circ f_{i}) \otimes \ldots \otimes f_n &\in   \mathbb{Z}[G/H,G/H_0]\otimes \mathbb{Z}[G/H_0,\ldots,\widehat{G/H_i},\ldots,G/H_n]\\ & \mbox{\;\;\; if \;} 0 < i < n; \\
                                                         f\otimes f_1 \otimes \ldots \otimes f_{n-1} &\in   \mathbb{Z}[G/H,G/H_0]\otimes \mathbb{Z}[G/H_0,\ldots,G/H_{n-1}]\\  &\mbox{\;\;\; if \;}  i=n.
                                                       \end{array}\right.$$}
                                                       \medskip
Finally, we define a natural transformation $\varepsilon : B_0(\mathfrak{F},G) \rightarrow \underline{\mathbb{Z}}$ by:  $$\varepsilon(G/H) : B_0(\mathfrak{F},G)(G/H) \rightarrow \mathbb{Z}: f \in \mathrm{Mor}(G/H,G/H_0)\mapsto 1.$$ It can be shown that:
\medskip
\[ \ldots \rightarrow B_n(\mathfrak{F},G) \xrightarrow{d_n} B_{n-1}(\mathfrak{F},G) \xrightarrow{d_{n-1}} \ldots \rightarrow B_1(\mathfrak{F},G) \xrightarrow{d_1} B_{0}(\mathfrak{F},G) \xrightarrow{\varepsilon} \underline{\mathbb{Z}} \rightarrow 0 \]
\medskip
is a free $\orb$-resolution of $\underline{\mathbb{Z}}$  (see \cite{Waner}). It follows that for each $n \in \mathbb{N}$ and every $M \in \orbmod$, we have:
\medskip
$$\mathrm{H}^n_{\mathfrak{F}}(G,M)= \mathrm{H}^n(\nathom(B_{\ast}(\mathfrak{F},G),M)).$$
\medskip

Next, we simplify the expression  for the abelian group $\nathom(B_{\ast}(\mathfrak{F},G),M)$ above, using a basis for the $B_{n}(\mathfrak{F},G)$. To this end, denote the product of sets:
\medskip
\[ \mathrm{Mor}(G/H_0,G/H_1)\times \mathrm{Mor}(G/H_1,G/H_2) \times \ldots \times \mathrm{Mor}(G/H_{n-1},G/H_n) \]
\medskip
by $\mathrm{Mor}(G/H_0,G/H_1,\ldots,G/H_n)$, for each $n \in \mathbb{N}$ and each $(H_0,H_1,\ldots, H_n) \in \mathfrak{F}^{n+1}$. Denote $id:G/H_0 \rightarrow G/H_0$ by $id^{\scriptscriptstyle H_0}$ for each $H_0 \in \mathfrak{F}$. Now, for each $n \in \mathbb{N}$ and each $H_0 \in \mathfrak{F}$, we define the set:
\medskip
\[ \Delta^n_{H_0} =\{ (id^{\scriptscriptstyle H_0},f_1,\ldots,f_n) \ | \ \forall (f_1,\ldots,f_n) \in  \mathrm{Mor}(G/H_0,G/H_1,\ldots,G/H_n), \forall (H_1,\ldots,H_n) \in \mathfrak{F}^n\}. \]
\medskip
\begin{lemma}\label{lemma: cochain complex} For each $n \in \mathbb{N}$, $(\Delta^n,\beta^n)$ is a basis for $B_n(\mathfrak{F},G)$, where $\Delta^n= \coprod_{H \in \mathfrak{F}}\Delta^n_{\scriptscriptstyle H}$ and $\beta^n: \Delta^n \rightarrow \mathfrak{F}: \delta \in \Delta^n_{\scriptscriptstyle H} \mapsto H$. Moreover, there is an isomorphism of abelian groups:
\medskip
\[\nathom(B_{n}(\mathfrak{F},G),M) \cong \mathrm{Hom}_{\mathfrak{F}\mbox{-set}}(\Delta^n,\mathfrak{F}(M)).\]
where $\mathrm{Hom}_{\mathfrak{F}\mbox{-set}}(\Delta^n,\mathfrak{F}(M))$ the group of maps of $\mathfrak{F}$-sets from $(\Delta^n,\beta^n)$ to $\mathfrak{F}(M)$.
\end{lemma}
\begin{proof} Fix a set $(H_0,\ldots,H_n)\in \mathfrak{F}^{n+1}$, denote $T=\mathbb{Z}[G/H_0,\ldots,G/H_n]$ and let $E=\{e_{\alpha}\}_{\alpha \in I}$ be an indexing for $\mathrm{Mor}(G/H_0,G/H_1,\ldots,G/H_n)$.  Clearly, $E$ is a basis for $T$. Define the set $\Pi=\coprod_{H \in \mathfrak{F}} \Pi_{\scriptscriptstyle H}$, with  $\Pi_{H_0}=\{(id^{\scriptscriptstyle H_0},e_{\alpha}) \ | \ \alpha \in I \}$ and $\Pi_{\scriptscriptstyle H}=\emptyset$ for $H \neq H_0$. From this we obtain a base function $\pi: \Pi \rightarrow \mathfrak{F}: \delta \in \Pi_{\scriptscriptstyle H} \mapsto H$. We claim that $\mathbb{Z}[-,G/H_0]\otimes \underline{T}$ is free on $(\Pi,\pi)$. Indeed, let $N$ be any $\orb$-module and let $f: (\Pi,\pi) \rightarrow \mathfrak{F}(N)$ be a map of $\mathfrak{F}$-sets. Suppose there exists a natural transformation $\bar{f}: \mathbb{Z}[-,G/H_0]\otimes \underline{T} \rightarrow N$ that extends $f$. Then for any $H \in \mathfrak{F}$ and any $\varphi \in \mathrm{Mor}(G/H,G/H_0)$, the following diagram commutes:
\[ \xymatrix{
\mathbb{Z}[G/H,G/H_0]\otimes T \ \ \ \ \ \ar[r]^{  \bar{f}(G/H)  } & \ \ \ \ \ N(G/H) \\
\mathbb{Z}[G/H_0,G/H_0]\otimes T \ar[u]^{\mathbb{Z}[\varphi,G/H_0]\otimes id} \ \ \ \ \ \ar[r]^{ \bar{f}(G/H_0)} &\ \ \ \ \  N(G/H_0). \ar[u]_{N(\varphi)} \\
\Pi_{H_0} \ar[u] \ar[ur]^{f_{H_0}} & \\ } \]
Following $(id^{\scriptscriptstyle H_0},e_{\alpha})\in \Pi_{H_0}$ through the diagram, we see that $\bar{f}(G/H)(\varphi,e_{\alpha})$ is completely determined by $f_{H_0}((id^{\scriptscriptstyle H_0},e_{\alpha}))$. This shows that $\bar{f}$ is unique if it exists. The diagram above also suggests how to define $\bar{f}$. For each $H \in \mathfrak{F}$, $\varphi \in \mathrm{Mor}(G/H,G/H_0)$, and  $e_{\alpha} \in E$, we define:
\medskip
$$\bar{f}(G/H)(\varphi \otimes e_{\alpha})= N(\varphi)(f_{H_0}((id^{\scriptscriptstyle H_0},e_{\alpha}))).$$
\medskip
Since $\mathbb{Z}[G/H,G/H_0]\otimes T$ is the free $\mathbb{Z}$-module generated by all $\varphi \otimes e_{\alpha}$, we can extend $\bar{f}$ to $\mathbb{Z}[G/H,G/H_0]\otimes T$. One easily verifies that $\bar{f}$, so defined, yields the desired natural transformation. We conclude that $\mathbb{Z}[-,G/H_0]\otimes \underline{T}$ is free on $(\Pi,\pi)$. It now immediately follows that
$B_n(\mathfrak{F},G)$
 is free on $(\Delta^n,\beta^n)$. \\
\indent  Note that $\mathrm{Hom}_{\mathfrak{F}\mbox{-set}}(\Delta^n,\mathfrak{F}(M))$ inherits a $\mathbb{Z}$-module structure from $M$ by pointwise addition. Since $B_n(\mathfrak{F},G)$ is a free $\orb$-module on $(\Delta^n,\beta^n)$, we have an isomorphism:
\[ \mathrm{Hom}_{\mathfrak{F}\mbox{-set}}(\Delta^n,\mathfrak{F}(M)) \rightarrow \nathom(B_{n}(\mathfrak{F},G),M) : f \mapsto \bar{f} \]
where a map of $\mathfrak{F}$-sets is mapped to the unique natural transformation that extends it.
\end{proof}
\begin{notation} \rm When $f$ is a map of  $\mathfrak{F}$-sets from $(\Delta^n,\beta^n)$ to $\mathfrak{F}(M)$, we will denote $f_{\scriptscriptstyle H}(id^{\scriptscriptstyle H},\varphi_1,\ldots,\varphi_n)$ by $f_{\scriptscriptstyle H}(\varphi_1,\ldots,\varphi_n)$ for all
$(\varphi_1,\ldots,\varphi_{n}) \in \mathrm{Mor}(G/H,G/H_1,\ldots,G/H_{n})$.
\end{notation}
\begin{definition} \rm Define the cochain complex:
\[ 0 \rightarrow C^0_{\mathfrak{F}}(G,M) \xrightarrow{d^0} C^1_{\mathfrak{F}}(G,M) \rightarrow \ldots \rightarrow C^{n}_{\mathfrak{F}}(G,M) \xrightarrow{d^{n}} C^{n+1}_{\mathfrak{F}}(G,M) \rightarrow \ldots  \]
by $C^{n}_{\mathfrak{F}}(G,M)=\mathrm{Hom}_{\mathfrak{F}\mbox{-set}}(\Delta^n,\mathfrak{F}(M)) $ for each $n$ and \[ d^n: C^n_{\mathfrak{F}}(G,M) \rightarrow C^{n+1}_{\mathfrak{F}}(G,M) : f \mapsto d^n(f),\] with
 \begin{eqnarray*} d^n(f)_{\scriptscriptstyle H}( \varphi_1,\ldots,\varphi_{n+1})  &=&   M(\varphi_1)\Big( f_{\scriptscriptstyle {H_1}}(\varphi_2,\ldots,\varphi_{n+1})\Big)  \\
                                                       & & +  \sum_{i=1}^n(-1)^if_{\scriptscriptstyle H}(\varphi_1,\ldots,\varphi_{i+1} \circ \varphi_i,\ldots,\varphi_{n+1}) \\
                                                        & & + (-1)^{n+1}f_{\scriptscriptstyle H}(\varphi_1,\ldots,\varphi_n) \end{eqnarray*}
 where $(\varphi_1,\ldots,\varphi_{n+1}) \in \mathrm{Mor}(G/H,G/H_1,\ldots,G/H_{n+1})$.
\end{definition}
\begin{theorem} \label{th: cochain complex} Let $G$ be a group, $\mathfrak{F}$ a family of subgroups of $G$ and $M$ an $\orb$-module. Then for each $n \in \mathbb{N}$ we have:
\[ \mathrm{H}^n_{\mathfrak{F}}(G,M) = \mathrm{H}^n(C^{\ast}_{\mathfrak{F}}(G,M)). \]
\end{theorem}
\begin{proof} This follows from the fact that the differentials of the complex $C^{\ast}_{\mathfrak{F}}(G,M)$ correspond to the differentials of the complex $\nathom(B_{\ast}(\mathfrak{F},G),M)$ under the isomorphism of Lemma \ref{lemma: cochain complex}.
\end{proof}
\begin{remark} \rm \label{remark: notation} In the next section, it will often be convenient to view a morphism of the orbit category $G/H \xrightarrow{x} G/K$ as the element $xK \in (G/K)^{\scriptscriptstyle H}$. With this notation, the differentials of the complex $C^{\ast}_{\mathfrak{F}}(G,M)$ become:
\begin{eqnarray*} d^n(f)_{\scriptscriptstyle H}(x_1H_1,\ldots,x_{n+1}H_{n+1})  &=&   M(x_1H_1)\Big( f_{\scriptscriptstyle H_1}(x_2H_2,\ldots,x_{n+1}H_{n+1})\Big)  \\
                                                       & & +  \sum_{i=1}^n(-1)^if_{\scriptscriptstyle H}(x_1H_1,\ldots,x_{i}x_{i+1} H_{i+1},\ldots,x_{n+1}H_{n+1}) \\
                                                        & & + (-1)^{n+1}f_{\scriptscriptstyle H}(x_1H_1,\ldots,x_{n}H_{n}) \end{eqnarray*}
where $M(x_1H_1)=M(G/H \xrightarrow{x_1} G/H_1)$.
\end{remark}

\section{Interpretations of Low dimensional Bredon cohomology} \label{sec: H1}
\subsection{Interpretation of $\mathrm{H}^0_{\mathfrak{F}}(G,M)$ }
Although it is well-known, we include an interpretation of $\mathrm{H}^0_{\mathfrak{F}}(G,M)$, for completeness.

Let $G$ be a group, $\mathfrak{F}$ a family of subgroups of $G$ and $M$ an $\orb$-module. It follows immediately from the derived functor construction of Bredon cohomology that $$\mathrm{H}^0_{\mathfrak{F}}(G,M)=\nathom(\underline{\mathbb{Z}},M).$$
Using this fact or Theorem \ref{th: cochain complex}, we obtain:
\[ \mathrm{H}^0_{\mathfrak{F}}(G,M) = \lim_{G/H \in \orb}M(G/K). \]
In fact, the functor: $$\nathom(\underline{\Z},-): \orbmod \rightarrow \mathfrak{Ab}$$ is naturally equivalent to the limit functor: $$\lim_{G/H \in \orb}: \orbmod \rightarrow \mathfrak{Ab}.$$ So, Bredon cohomology can also be seen as the right derived functors of this functor.

\subsection{$\mathfrak{F}$-Derivations and Principal $\mathfrak{F}$-Derivations}
Let $G$ be a group, $\mathfrak{F}$ a family of subgroups of $G$ and $M$ a $\orb$-module.
\begin{definition} \rm
An $\mathfrak{F}$-derivation of $G$ into $M$ is a map of $\mathfrak{F}$-sets $\mathcal  D: \Delta^1 \rightarrow \mathfrak{F}(M)$ such that: $\forall H,H_1,H_2 \in \mathfrak{F}, \forall x_1H_1 \in (G/H_1)^{\scriptscriptstyle H}, \forall x_2H_2 \in (G/H_2)^{H_1}$:
\[ \mathcal D_{\scriptscriptstyle H}(x_1x_2H_2)= M(x_1H_1)\Big(\mathcal D_{H_1}(x_2H_2)\Big) + \mathcal D_{\scriptscriptstyle H}(x_1H_1). \]
The set of all $\mathfrak{F}$-derivations of $G$ into $M$ forms an abelian group which we denote by $\mathrm{Der}_{\mathfrak{F}}(G,M)$. \\
Given an element $m=(m_{\scriptscriptstyle H})_{H \in \mathfrak{F}} \in \prod_{H \in \mathfrak{F}}M(G/H)$, the principal $\mathfrak{F}$-derivation $\mathcal D_m$ of $G$ into $M$ is a map of $\mathfrak{F}$-sets $\mathcal D: \Delta^1 \rightarrow \mathfrak{F}(M)$ defined as follows: $\forall H,K \in \mathfrak{F}, \forall xK \in (G/K)^{\scriptscriptstyle H}$:
\[(\mathcal  D_m)_{\scriptscriptstyle H}(xK) = M(xK)(m_{\scriptscriptstyle K})-m_{\scriptscriptstyle H}. \]
The set of all principal $\mathfrak{F}$-derivation of $G$ into $M$ forms an abelian subgroup of $\mathrm{Der}_{\mathfrak{F}}(G,M)$ which we denote by $\mathrm{PDer}_{\mathfrak{F}}(G,M)$.
\end{definition}
We can now prove the following.
\begin{theorem} \label{th: H^1 int}If $G$ is a group, $\mathfrak{F}$ a family of subgroups of $G$ and $M$ a $\orb$-module, then we have
\[ \mathrm{H}^1_{\mathfrak{F}}(G,M) \cong \mathrm{Der}_{\mathfrak{F}}(G,M)\Big/ \mathrm{PDer}_{\mathfrak{F}}(G,M). \]
\end{theorem}
\begin{proof}
This is immediate from the definition of (principal) $\mathfrak{F}$-derivations, Theorem \ref{th: cochain complex} and Remark \ref{remark: notation}.
\end{proof}
We shall now restrict to fixed point functors.
\begin{definition} \rm
Suppose $M \in G\mbox{-mod}$.
Given an element $m \in M$, a principal derivation $D_m$ of $G$ into $M$ is a set map $D: G \rightarrow M$ defined as follows: for all $x \in G$:
\[ D_m(x) = x\cdot m-m. \] An $\mathfrak{F}$-derivation of $G$ into $M$ is a set map $D : G \rightarrow M$ such that: for all $x,y \in G$:
\[ D(xy)= x\cdot D(y)+D(x) \]
and $D|{_{\scriptscriptstyle H}}: H \rightarrow M$ is a principal derivation $D_{m_{\scriptscriptstyle H}}:H\to M$, for each $H \in \mathfrak{F}$. Note that $m_{\scriptscriptstyle H}$ depends on $H$ and that
the set of all $\mathfrak{F}$-derivations of $G$ into $M$ forms an abelian group which we denote by $\mathrm{Der}_{\mathfrak{F}}(G,M)$. The set of all principal derivations of $G$ into $M$ forms an abelian subgroup of $\mathrm{Der}_{\mathfrak{F}}(G,M)$ which we denote by $\mathrm{PDer}_{\mathfrak{F}}(G,M)$.
\end{definition}
We now obtain the following
\begin{proposition} If $G$ is a group, $\mathfrak{F}$ a family of subgroups of $G$ that contains the trivial subgroup and $M \in G\mbox{-mod}$, then there exists an isomorphism:
\[ \mathrm{Der}_{\mathfrak{F}}(G,M)/ \mathrm{PDer}_{\mathfrak{F}}(G,M) \cong \mathrm{Der}_{\mathfrak{F}}(G,\underline{M})/\mathrm{PDer}_{\mathfrak{F}}(G,\underline{M}).\]
\end{proposition}
\begin{proof} Consider the map:
\[ \overline{\Psi} : \mathrm{Der}_{\mathfrak{F}}(G,M) \rightarrow \mathrm{Der}_{\mathfrak{F}}(G,\underline{M})/\mathrm{PDer}_{\mathfrak{F}}(G,\underline{M}): f \mapsto \Psi(f)+\mathrm{PDer}_{\mathfrak{F}}(G,\underline{M}), \]
with $\Psi(f)_{\scriptscriptstyle H}(xK)=m_{\scriptscriptstyle H}-x\cdot m_{\scriptscriptstyle K} +f(x)$ for all $H,K \in \mathfrak{F},$ and all $xK \in (G/K)^{\scriptscriptstyle H}$, where $(m_{\scriptscriptstyle H})_{H \in \mathfrak{F}}$ is collection of elements in $M$ such that $f|{_{\scriptscriptstyle H}}: H \rightarrow M$ is the principal derivation $D_{m_{\scriptscriptstyle H}}$ for each $H \in \mathfrak{F}$. Via elementary calculations one can check that $\Psi(f) \in  \mathrm{Der}_{\mathfrak{F}}(G,\underline{M})$.

Note that the chosen elements $(m_{\scriptscriptstyle H})_{H \in \mathfrak{F}}$ are not unique. However, if we choose a different set of elements $(n_{\scriptscriptstyle H})_{H \in \mathfrak{F}}$ such that $f|{_{\scriptscriptstyle H}}: H \rightarrow M$ is the principal derivation $D_{n_{\scriptscriptstyle H}}$ for each $H \in \mathfrak{F}$, then it follows: $$(m_{\scriptscriptstyle H}-n_{\scriptscriptstyle H})_{H \in \mathfrak{F}} \in \prod_{H \in \mathfrak{F}}\underline{M}(G/H)$$ for each $H \in \mathfrak{F}$. Therefore, if we define $\Psi(f)_{\scriptscriptstyle H}(xK)$ using $(n_{\scriptscriptstyle H})_{H \in \mathfrak{F}}$ instead, then this will give the same $\mathfrak{F}$-derivation up to an element in $\mathrm{PDer}_{\mathfrak{F}}(G,\underline{M})$. This implies that $\overline{\Psi}$ is a well-defined $\mathbb{Z}$-module homomorphism.

We claim that $\Psi$ is surjective.
Let $\mathcal  D \in \mathrm{Der}_{\mathfrak{F}}(G,\underline{M})$ and consider arbitrary $H,K \in \mathfrak{F}$ and $xK \in (G/K)^{\scriptscriptstyle H}$. Using the composition: $$G/\{e\} \xrightarrow{e} G/H \xrightarrow{x} G/K = G/\{e\} \xrightarrow{x} G/K,$$ we compute: $$\mathcal  D_{\{e\}}(xK) = \underline{M}(H)(\mathcal  D_{\scriptscriptstyle H}(xK))+\mathcal D_{\{e\}}(H) = \mathcal D_{\scriptscriptstyle H}(xK)+ \mathcal D_{\{e\}}(H)$$ and by considering the composition: $$G/\{e\} \xrightarrow{x} G/\{e\} \xrightarrow{e} G/K=  G/\{e\} \xrightarrow{x} G/K,$$ we compute: $$\mathcal D_{\{e\}}(xK) = \underline{M}(x\{e\})\mathcal D_{\{e\}}(K) + \mathcal D_{\{e\}}(x\{e\})
= x\cdot \mathcal D_{\{e\}}(K)+\mathcal D_{\{e\}}(x\{e\}).$$ Combining these equations, entails:
\begin{eqnarray}
\mathcal D_{\scriptscriptstyle H}(xK) & = & \mathcal D_{\{e\}}(x\{e\}) + x\cdot \mathcal D_{\{e\}}(K) - \mathcal D_{\{e\}}(H) \label{eq: derivation}
\end{eqnarray}
for all $H,K \in \mathfrak{F}, \forall xK \in (G/K)^{\scriptscriptstyle H}$. Now, define: $$f: G \rightarrow M: x \mapsto \mathcal D_{\{e\}}(x\{e\}).$$ Then  $f(xy)=x\cdot f(y)+f(x)$ for all $x,y \in G$. Let $h \in H$ for some arbitrary $H \in \mathfrak{F}$ and consider the composition: $$G/\{e\} \xrightarrow{h} G/\{e\} \xrightarrow{e} G/H = G/\{e\} \xrightarrow{e} G/H.$$ We compute: $$\mathcal D_{\{e\}}(H) = \underline{M}(h)(\mathcal D_{\{e\}}(H))+\mathcal D_{\{e\}}(h\{e\})
= h\cdot \mathcal D_{\{e\}}(H)+\mathcal D_{\{e\}}(h\{e\}).$$ Define the elements $(m_{\scriptscriptstyle H})_{H \in \mathfrak{F}}$ in $M$ as $m_{\scriptscriptstyle H}=-\mathcal D_{\{e\}}(H)$ for all $H \in \mathfrak{F}$.
This implies that $f(h)=h\cdot m_{\scriptscriptstyle H}-m_{\scriptscriptstyle H}$ for all  $H \in \mathfrak{F}$ and hence, $f \in \mathrm{Der}_{\mathfrak{F}}(G,M)$.
Using equation (\ref{eq: derivation}), it follows: $$\mathcal  D_{\scriptscriptstyle H}(xK)  =  m_{\scriptscriptstyle H}-x\cdot m_{\scriptscriptstyle K}+f(x) \mbox{\; for all \;} H,K \in \mathfrak{F}, \forall xK \in (G/K)^{\scriptscriptstyle H}.$$ Thus, $\mathcal D+ \mathrm{PDer}_{\mathfrak{F}}(G,\underline{M})=\overline{\Psi}(f)$. This shows that $\Psi$ is surjective.

Suppose now $f$ is a principal derivation for $m \in M$, i.e. $f(x)=x\cdot m - m$ for all $x \in G$.
To determine $\Psi(f)$, we can set all $m_{\scriptscriptstyle H}$ equal to $m$. Hence, it follows that $\overline{\Psi}(f)=0$. This implies that $\mathrm{PDer}_{\mathfrak{F}}(G,M)$ is contained in the kernel of $\overline{\Psi}$.

Finally, suppose $\overline{\Psi}(f)=0$ for some $f \in \mathrm{Der}_{\mathfrak{F}}(G,M)$. This implies that $\Psi(f)$ (where $(m_{\scriptscriptstyle H})_{H \in \mathfrak{F}}$ are the chosen elements in $M$ such that $f_{|H}: H \rightarrow M$ is the principal derivation $D_{m_{\scriptscriptstyle H}}$) is a principal $\mathfrak{F}$-derivation $\mathcal D_n$ with $n=(n_{\scriptscriptstyle H})_{H \in \mathfrak{F}} \in \prod_{H \in \mathfrak{F}}M^{\scriptscriptstyle H}$. It follows that: $$f(x)=x\cdot (n_{\scriptscriptstyle K}+m_{\scriptscriptstyle K})-(n_{\scriptscriptstyle H}+m_{\scriptscriptstyle H})$$ for all $H,K \in \mathfrak{F}, \forall xK \in (G/K)^{\scriptscriptstyle H}$. Setting $x=e$ and  $H=\{e\}$ and considering and arbitrary subgroup $K \in \mathfrak{F}$, this implies: $$0=f(e)=(n_{\scriptscriptstyle K}+m_{\scriptscriptstyle K})-(n_e+m_e).$$ Hence, $n_e+m_e=n_{\scriptscriptstyle K}+m_{\scriptscriptstyle K}$ for all $K \in \mathfrak{F}$. We conclude that: $$f(x)=x\cdot (n_e+m_e)-(n_e+m_e)$$ for all $x \in G$. This shows that the kernel of $\overline{\Psi}$ is contained in $\mathrm{PDer}_{\mathfrak{F}}(G,M)$.

Considering everything together, we have an isomorphism of $\mathbb{Z}$-modules:
\[  \mathrm{Der}_{\mathfrak{F}}(G,M)/ \mathrm{PDer}_{\mathfrak{F}}(G,M) \cong \mathrm{Der}_{\mathfrak{F}}(G,\underline{M})/ \mathrm{PDer}_{\mathfrak{F}}(G,\underline{M}):  \bar{f}  \mapsto \overline{\Psi}(f) . \]
\end{proof}

\begin{corollary}\label{cor: special int H^1} If $G$ is a group and $\mathfrak{F}$ is a family of subgroups of $G$ that contains the trivial subgroup and $M \in G\mbox{-mod}$, then we have:
\[ \mathrm{H}^1_{\mathfrak{F}}(G,\underline{M}) \cong \mathrm{Der}_{\mathfrak{F}}(G,M)/ \mathrm{PDer}_{\mathfrak{F}}(G,M). \]
\end{corollary}
\begin{proof} This is immediate from Theorem \ref{th: H^1 int} and the preceding proposition.
\end{proof}
\begin{corollary} \label{cor: H^1 intersect} If $G$ is a group and $\mathfrak{F}$ is a family of subgroups of $G$ that contains the trivial subgroup and $M \in G\mbox{-mod}$, then we have:
\[ \mathrm{H}^1_{\mathfrak{F}}(G,\underline{M}) \cong   \cap_{H \in \mathfrak{F}}\mathrm{Ker}\Big( i^1_{\scriptscriptstyle H}: \mathrm{H}^1(G,M) \rightarrow \mathrm{H}^1(H,M)\Big),\] where $i^1_{\scriptscriptstyle H}$ is induced by the inclusion map $i_{\scriptscriptstyle H} : H \hookrightarrow G$.
\end{corollary}
\begin{proof} This follows straight from the well-known interpretation of $1$-dimensional cohomology of groups in terms of derivations and principal derivations (e.g. see \cite{Brown}), the preceding corollary and the definitions of $\mathrm{Der}_{\mathfrak{F}}(G,M)$ and $\mathrm{PDer}_{\mathfrak{F}}(G,M)$.
\end{proof}
\begin{corollary} \label{cor: finite group Q}If $G$ is a finite group and $\mathfrak{F}$ is a family of subgroups of $G$ that contains the trivial subgroup, then
\[ \mathrm{H}^2_{\mathfrak{F}}(G,\underline{\mathbb{Z}})\cong \{ f \in \mathrm{Hom}(G,\mathbb{Q}/\mathbb{Z}) \ | \ f(H)=0 \ \mbox{for all} \ H \in \mathfrak{F}\}. \]
\end{corollary}
\begin{proof} The short exact sequence of trivial fixed point functors $0\to \underline{\mathbb{Z}} \rightarrow \underline{\mathbb{Q}} \rightarrow \underline{\mathbb{Q}/\mathbb{Z}}\to 0$ gives rise to a long exact Bredon cohomology sequence:
\[ \ldots \rightarrow \mathrm{H}^1_{\mathfrak{F}}(G,\underline{\mathbb{Q}}) \rightarrow  \mathrm{H}^1_{\mathfrak{F}}(G,\underline{\mathbb{Q}/\mathbb{Z}}) \rightarrow \mathrm{H}^2_{\mathfrak{F}}(G,\underline{\mathbb{Z}}) \rightarrow \mathrm{H}^2_{\mathfrak{F}}(G,\underline{\mathbb{Q}}) \rightarrow \ldots \ . \]
Since the order of $G$ is invertible in $\mathbb{Q}$, it follows that $\mathrm{H}^1(G,\mathbb{Q})$ and $\mathrm{H}^2(G,\mathbb{Q})$ are zero (e.g. see \cite{Brown}). By the preceding corollary, this implies that $\mathrm{H}^1_{\mathfrak{F}}(G,\underline{\mathbb{Q}})=0$.
Since $\mathrm{H}^1(H,\mathbb{Q})=0$ for all finite groups $H$, it follows from  Proposition \ref{prop: ker ext} and Corollary \ref{cor: intersection} that $\mathrm{H}^2_{\mathfrak{F}}(G,\underline{\mathbb{Q}}) \subseteq \mathrm{H}^2(G,\mathbb{Q})=0$. Hence, by the long exact sequence $\mathrm{H}^1_{\mathfrak{F}}(G,\underline{\mathbb{Q}/\mathbb{Z}}) \cong \mathrm{H}^2_{\mathfrak{F}}(G,\underline{\mathbb{Z}})$. Using Corollary \ref{cor: special int H^1} and the fact that $\mathbb{Q}/\mathbb{Z}$ is a trivial $G$-module, it is not difficult to verify that:
\[ \mathrm{H}^1_{\mathfrak{F}}(G,\underline{\mathbb{Q}/\mathbb{Z}})\cong \{ f \in \mathrm{Hom}(G,\mathbb{Q}/\mathbb{Z}) \ | \ f(H)=0 \ \mbox{for all} \ H \in \mathfrak{F}\}.\]
This concludes the proof.
\end{proof}
\subsection{Split extensions}
Let $G$ be a group, $\mathfrak{F}$ be a family of subgroups of $G$ and $M$ be a $\orb$-module. We have seen that there exists, up to equivalence of extensions, a unique split abelian right extension of $\orb$ by $M$  such that the induced $\orb$-module structure on $M$ by the extension equals the original one. However, a splitting of this split extensions does not have to be unique. In view of Proposition \ref{lemma: unique split ext}, let us consider the standard split extension:
\[ \mathfrak{M} \xrightarrow{i} \mathfrak{M} \rtimes \orb \xrightarrow{\pi} \orb \]
where $\mathfrak{M}$ is the usual category with objects $M(G/H)$. An obvious splitting for this extension is the functor $s: \orb \rightarrow \mathfrak{M} \rtimes \orb$ with $$s(G/H)=M(G/H)\times G/H \mbox{\; and \;} s(\varphi)=({0},\varphi)$$ for $\varphi \in \mathrm{Mor}(G/H,G/K).$ But in fact, any functor $t: \orb \rightarrow \mathfrak{M} \rtimes \orb$ with $$t(G/H)=M(G/H)\times G/H \mbox{\; and \;} t(\varphi)=(\overline{\mathcal D(\varphi)},\varphi)$$ for $\varphi \in \mathrm{Mor}(G/H,G/K)$ is a splitting. (Recall that  $\overline{\mathcal D(\varphi)} \in \mathrm{Mor}(M(G/H), M(G/H))$, i.e. $\mathcal D(\varphi) \in M(G/H)$). Imposing that $t$ is a functor is equivalent to requiring that the relation:
\[ \mathcal D(\psi \circ \varphi)=\mathcal D(\varphi) + M(\varphi)(\mathcal D(\psi)) \]
holds, for all $\varphi \in \mathrm{Mor}(G/H,G/H_1)$ and for all $\psi \in \mathrm{Mor}(G/H_1,G/H_2)$. To see this, just translate $s(\psi \circ \varphi)=s(\psi)\circ s(\varphi)$ into a statement about $\mathcal D$ using the composition law in $\mathfrak{M} \rtimes \orb$.

Note that  $\mathcal D$ can be considered as a map of $\mathfrak{F}$-sets from $\Delta^1$ to $\mathfrak{F}(M)$. Then, the equation above defines $\mathcal D$ as  an $\mathfrak{F}$-derivation of $G$ into $M$. This shows that there is a bijection between splittings of the standard split extension of $\orb$ with $M$ and the derivations of $G$ into $M$.
\begin{definition} \rm  Two splittings $s$ and $t$ of the extension:
\[ \mathfrak{M} \xrightarrow{i} \mathfrak{M} \rtimes \orb \xrightarrow{\pi} \orb \]
are said to be $M$-conjugate if there exists an element $m=(m_{\scriptscriptstyle H})_{H \in \mathfrak{F}} \in \prod_{H \in \mathfrak{F}}M(G/H)$ such that
\[  s(\varphi)\circ i(\bar{m}_{\scriptscriptstyle H})=i(\bar{m}_{\scriptscriptstyle K})\circ t(\varphi)  \]
for all $H,K \in \mathfrak{F}$, for all $\varphi \in \mathrm{Mor}(G/H,G/K)$. The notion of $M$-conjugacy places an equivalence relation on the set of all splittings.
\end{definition}
If $\mathcal D$ and $\mathcal D'$ are the $\mathfrak{F}$-derivations corresponding to $s$ and $t$, then $M$-conjugacy is equivalent to the relation:
\[\mathcal D_{\scriptscriptstyle H}(xK)-\mathcal D'_{\scriptscriptstyle H}(x_{\scriptscriptstyle K})= M(xK)(m_{\scriptscriptstyle K})-m_{\scriptscriptstyle H} = (\mathcal D_m)_{\scriptscriptstyle H}(x_{\scriptscriptstyle K})\]
for all $H,K \in \mathfrak{F}$, for all $xK \in (G/K)^{\scriptscriptstyle H}$. This means that two splittings are $M$-conjugate if and only if their corresponding $\mathfrak{F}$-derivations differ by a principal derivation. Hence, there is a one-to-one correspondence between the $M$-conjugacy classes of splittings of the standard abelian right extension of $\orb$ by $M$ and $\mathrm{Der}_{\mathfrak{F}}(G,M)/ \mathrm{PDer}_{\mathfrak{F}}(G,M)$. Using Theorem \ref{th: H^1 int}, we arrive at the following interpretation for $\mathrm{H}^1_{\mathfrak{F}}(G,M)$.
\begin{theorem} \label{th: H^1 int split ext}If $G$ is a group, $\mathfrak{F}$ a family of subgroups of $G$ and $M$ a $\orb$-module then there is a one-to-one correspondence between the $M$-conjugacy classes of splittings of
\[ \mathfrak{M} \xrightarrow{i} \mathfrak{M} \rtimes \orb \xrightarrow{\pi} \orb \]
and the elements of $\mathrm{H}^1_{\mathfrak{F}}(G,M)$. In particular, $\mathrm{H}^1_{\mathfrak{F}}(G,M)=0$ if and only if all splittings of the standard abelian right extensions of $G$ by $M$ are $M$-conjugate.
\end{theorem}
Let us again treat the case where the $\orb$-module is a fixed point functor $\underline{M}$. Consider the standard split $\mathfrak{F}$-structure on $(G,M)$:
\[ 0 \rightarrow M \rightarrow \Gamma = M \rtimes G \xrightarrow{\pi} G \rightarrow 1 \]
with $\Gamma_{\scriptscriptstyle H}=({0},H)$ for each $H \in \mathfrak{F}$. A map $s: G \rightarrow M \rtimes G$ is a splitting for this $\mathfrak{F}$-structure if and only if $s$ is a group homomorphism such that $\pi \circ s= id$ and $s(H)=(-m_{\scriptscriptstyle H},e)({0},H)(m_{\scriptscriptstyle H},e)$ for elements $(m_{\scriptscriptstyle H})_{H \in \mathfrak{F}}$, for each $H \in \mathfrak{F}$. This is equivalent to saying that $s$ is of the form: $$s: G \rightarrow M \rtimes G: x \mapsto (D(x),x),$$ where $D$ is an $\mathfrak{F}$-derivation of $G$ into M. Hence, the set of splittings of the standard split $\mathfrak{F}$-structure on $(G,M)$ is in bijective correspondence with the $\mathfrak{F}$-derivations of $G$ into M.
\begin{definition} \rm  Two splittings $s$ and $t$ of the standard split $\mathfrak{F}$-structure on $(G,M)$
are said to be $M$-conjugate if there exists an element $m \in M$ such that
\[  s(x)=mt(x)m^{\scriptscriptstyle -1}  \]
for all $x \in G$. The notion of $M$-conjugacy places an equivalence relation on the set of all splittings.
\end{definition}
One can readily verify that two splittings are $M$-conjugate if and only if their corresponding $\mathfrak{F}$-derivations differ by a principal derivation of $G$ into $M$. Using Corollary \ref{cor: special int H^1}, we arrive at the following result.
\begin{corollary}  If $G$ is a group, $\mathfrak{F}$ a family of subgroups of $G$ that contains the trivial subgroup and $M \in G\mbox{-mod}$, then there is a one-to-one correspondence between the $M$-conjugacy classes of splittings of the standard split $\mathfrak{F}$-structure on $(G,M)$ and the elements of $\mathrm{H}^1_{\mathfrak{F}}(G,\underline{M})$.  In particular, $\mathrm{H}^1_{\mathfrak{F}}(G,\underline{M})=0$ if and only if all splittings of the standard split $\mathfrak{F}$-structure on $(G,M)$ are $M$-conjugate.
\end{corollary}
\subsection{The second Bredon cohomology} \label{sec: H2}
Here, we present several applications of Theorem \ref{th: special ext} and the following  special case of a result of Hoff (see \cite{Hoff},\cite{Hoff2}).
\begin{theorem}[Hoff] \label{th: H^2} Suppose $G$ is a group, $\mathfrak{F}$ a family of subgroups of $G$ and $M$ a $\orb$-module. There is a one-to-one correspondence:
\[ \mathrm{H}^2_{\mathfrak{F}}(G,M) \leftrightarrow \mathrm{Ext}_{\mathfrak{F}}(G,M) \]
such that the zero element in $\mathrm{H}^2_{\mathfrak{F}}(G,M)$ corresponds to the class of split extensions in $\mathrm{Ext}_{\mathfrak{F}}(G,M)$.
\end{theorem}
Using Theorem \ref{th: special ext}, we immediately obtain the following
\begin{theorem} \label{cor: H^2 int} Let $G$ be a group,  $\mathfrak{F}$ be a family of subgroups of $G$ that contains the trivial subgroup, and  $M \in  G\mbox{-mod}$. There is a one-to-one correspondence between $\mathrm{H}^2_{\mathfrak{F}}(G,\underline{M})$ and $\mathrm{Str}_{\mathfrak{F}}(G,M)$, such that the zero element in $\mathrm{H}^2_{\mathfrak{F}}(G,\underline{M})$ corresponds to the class of split $\mathfrak{F}$-structures on $(G,M)$. In particular, $\mathrm{H}^2_{\mathfrak{F}}(G,\underline{M})=0$ if and only if every $\mathfrak{F}$-structure on $(G,M)$ splits.
\end{theorem}
\smallskip
\begin{corollary}\label{cor: intersection} Let $G$ be a group,  $\mathfrak{F}$ be a family of subgroups of $G$ that contains the trivial subgroup, and $\overline{\mathfrak{F}}$ be the smallest family of subgroups of $G$ containing $\mF$ that is closed under conjugation. If $H^1(H,M)=0$ for all $H \in \overline{\mathfrak{F}}$, then $$\mathrm{H}^2_{\mathfrak{F}}(G,\underline{M} )\cong  \cap_{H \in \mathfrak{F}}\mathrm{Ker}\Big( i^2_{\scriptscriptstyle H}: \mathrm{H}^2(G,M) \rightarrow \mathrm{H}^2(H,M)\Big),$$ where $i^2_{\scriptscriptstyle H}$ is induced by the inclusion map $i_{\scriptscriptstyle H} : H \hookrightarrow G$.
\end{corollary}
\begin{proof} Applying Proposition \ref{prop: ker ext} and the preceding theorem proves the result.
\end{proof}
\smallskip
\begin{corollary} Let $G$ be a virtually free group and let $\mathfrak{F}$ be the family of finite subgroups of $G$.  Let $M \in G\mbox{-mod}$ and consider an abelian extension of $G$ by $M$:
\[ 0 \rightarrow M \rightarrow \Gamma \rightarrow G \rightarrow 1. \]
Then every $\mathfrak{F}$-structure on this extension is split. In particular, if $H^1(H,M)=0$ for all $H \in \mathfrak{F}$, then this extension splits if and only if it splits when restricted to every $H \in \mathfrak{F}$.
\end{corollary}
\begin{proof} Since $G$ has a free subgroup of finite index, a result of Dunwoody (see \cite{Dunwoody}) implies that there exists a one-dimensional model for $E_{\mathfrak{F}}G$. This implies that $\mathrm{H}^2_{\mathfrak{F}}(G,N)=0$ for all $N \in \orbmod$ (e.g. see \cite{Nucinkis}). In particular, $\mathrm{H}^2_{\mathfrak{F}}(G,\underline{M})=0$ for all $M \in G\mbox{-mod}$. The statements now follow from the two previous results.
\end{proof}
\begin{corollary} Let $G$ be a countable group and let $\mathfrak{F}$ be the family of finitely generated subgroups of $G$.  Let $M \in G\mbox{-mod}$ and consider an abelian extension of $G$ by $M$:
\[ 0 \rightarrow M \rightarrow \Gamma \rightarrow G \rightarrow 1. \]
Then every $\mathfrak{F}$-structure on this extension is split. In particular, if $H^1(H,M)=0$ for all $H \in \mathfrak{F}$, then this extension splits if and only if it splits when restricted to every $H \in \mathfrak{F}$.
\end{corollary}
\begin{proof} We can write $G$ as a countable directed union $G=\bigcup_{n\in \mathbb N}G_n$ where $G_n\leq G_{n+1}$ and $G_n$ is finitely generated for each $n\in \mathbb N$. By Bass-Serre theory, $G$ acts on a tree $X$ with stabilizer subgroups the collection $\{G_n\}_{n\in \mathbb N}$. Then $X$ is a one-dimensional model for $E_{\mathfrak{F}}G$. Now, a similar argument as in the proof of the previous corollary finishes the proof.
\end{proof}
\section{Bredon-Galois cohomology}
Throughout this section, we assume some familiarity with the basic notions and results of Galois theory and discrete valuation theory (see \cite{Serre2},\cite{ZariskiSamuel}).

In Galois cohomology, one aims to understand the properties of a finite Galois extensions of fields $E / K$ by investigating the cohomology of the Galois group $\mathrm{Gal}(E/K)$ with coefficients in multiplicative group $E^{\times}$, where $G$ acts by field automorphisms. We propose a new invariant called \emph{Bredon-Galois cohomology} as a way to study the properties of a collection of intermediate fields of a finite Galois extensions $E / K$ by considering the $\mathfrak{F}$-Bredon cohomology of $\mathrm{Gal}(E/K)$ with coefficients in the fixed point functor $\underline{E}^{\times}$, where $\mathfrak{F}$ is the family of subgroups of $\mathrm{Gal}(E/K)$ corresponding under Galois correspondence to the given collection of intermediate fields. A first indication that this is an interesting invariant is that, under very mild conditions, there is analog of Hilbert's Theorem $90$ (e.g. see \cite{Serre2},\cite{Weibel}).
\begin{lemma} \label{lemma: bredon-hilbert}Let $E/K$ be a finite Galois extension of fields with Galois group $G$ and let $\mathfrak{F}$ be a family of subgroups of $G$ containing the trivial subgroup, then $\mathrm{H}^1_{\mathfrak{F}}(G,\underline{E}^{\times})=0$.
\end{lemma}
\begin{proof} By Corollary \ref{cor: H^1 intersect}, we known that $\mathrm{H}^1_{\mathfrak{F}}(G,\underline{E}^{\times})\subseteq \mathrm{H}^1(G,E^{\times})$. The statement now follows from Hilbert's Theorem $90$.
\end{proof}
From  Corollary \ref{cor: intersection} and Hilbert's Theorem $90$, we can immediately deduce the following.
\begin{lemma} \label{lemma: 2nd-bredon-galois}For any finite Galois extension of fields $E/K$ with Galois group $G$ and family of subgroups $\mathfrak{F}$ containing the trivial subgroup, there is an isomorphism: \[ \mathrm{H}^2_{\mathfrak{F}}(G,\underline{E}^{\times})\cong  \cap_{H \in \mathfrak{F}}\mathrm{Ker}\Big( \mathrm{H}^2(G,E^{\times}) \rightarrow \mathrm{H}^2(H,E^\times)\Big).\]
\end{lemma}
Hence, one can express the $2$-dimensional Bredon-Galois cohomology in terms of the $2$-dimensional Galois cohomology. One could therefore argue that, in low dimensions, it is unnecessary to introduce Bredon-cohomology as a new invariant. In what follows, we shall try to convince the reader to the contrary.

We use Bredon-Galois cohomology to study relative Brauer groups associated to field extensions.  In Section \ref{sec: short exact seq}, we construct two short exact sequences containing relative Brauer groups. It will be clear from the proof that the obstruction to the existence of these short exact sequences is a third-dimensional Bredon cohomology group. Hence, it is not clear how to obtain these exact sequence using Galois cohomology. This is a first motivation for the use of Bredon-Galois cohomology.

In Section \ref{sec: non-zero}, we expand some known methods for constructing non-zero elements in certain relative Brauer groups. The results in this section can also be proven using Galois cohomology. However, the use of Bredon-Galois cohomology hides a lot of the technical difficulties that would arise otherwise. Hence, our goal in this section is mainly to illustrate that Bredon-Galois cohomology is a more natural tool in this setting.

Let us recall the definition of the Brauer group associated to a field.
Let $k$ be a field and let $A$ and $A'$ be two simple central $k$-algebras. By Artin-Wedderburn Theorem, $A$ is isomorphic to a matrix algebra $M_n(D)$ and $A'$ is isomorphic to a matrix algebra $M_{n'}(D')$, where $D$ and $D'$ are both division $k$-algebras with center $k$. We say $A$ and $A'$ are \emph{equivalent} if $D$ and $D'$ are isomorphic. The equivalence class of $A$ is denoted by $[A]$. Let $\mathrm{Br}(k)$ the set of equivalence classes of central simple $k$-algebras. The operation: $$[A]+[B]:=[A\otimes_k B] \mbox{\; for \;} [A],[B] \in \mathrm{Br}(k),$$ turns $\mathrm{Br}(k)$ into an abelian group. The group $\mathrm{Br}(k)$ is called the \emph{Brauer group} of $k$. Given a homomorphism of fields $i: k \rightarrow K$, we can construct a group homomorphism: $$B(i): B(k) \rightarrow B(K): [A] \mapsto [A\otimes_k K],$$ where $K$ becomes a $k$-algebra via $i$. This construction turns $\mathrm{Br}(-)$ into a covariant functor from the category of fields to the category of abelian groups. If $k \subseteq K$ is an inclusion of fields then the \emph{relative Brauer group} $\mathrm{Br}(K/k)$ is by definition the kernel of the map $\mathrm{Br}(k) \rightarrow \mathrm{Br}(K)$, induced by the inclusion $k \hookrightarrow K$.

Let $E/K$ be a finite Galois extension and denote the Galois group of this extension by $G$. Let $E^{\times}$ be the multiplicative group of $E$. This group becomes a $G$-module under the action of $G$ by field automorphisms. Let $H$ be a subgroup of $G$ and let $L$ be the fixed field of $H$. By considering the following commutative diagram with exact rows, where the maps are induced by field inclusions:
\[ \xymatrix{0 \ar[r] & \mathrm{Br}(E/K) \ar[d]^{\phi} \ar[r] & \mathrm{Br}(K) \ar[d] \ar[r]& \mathrm{Br}(E) \ar[d]^{id} \\
0 \ar[r] & \mathrm{Br}(E/L) \ar[r]  & \mathrm{Br}(L)  \ar[r]& \mathrm{Br}(E), }\]
\smallskip
it follows that $\mathrm{ker}(\phi)\cong \mathrm{Br}(L/K)$. It is well-known that $\mathrm{H}^2(G,E^{\times})\cong \mathrm{Br}(E/K)$ and $\mathrm{H}^2(H,E^{\times})\cong \mathrm{Br}(E/L)$ such that the map $\phi$ corresponds to the restriction homomorphism $\mathrm{H}^2(G,E^{\times}) \rightarrow \mathrm{H}^2(H,E^{\times})$. We conclude that $\mathrm{Br}(L/K)$ is isomorphic to the kernel of the restriction map $\mathrm{H}^2(G,E^{\times}) \rightarrow \mathrm{H}^2(H,E^{\times})$.
\begin{proposition} \label{prop: bredon-galois} Let $E/K$ be a finite Galois extension with Galois group $G$. Let $\mathfrak{F}$ be any family of subgroups of $G$ containing the trivial subgroup, then
\[ \mathrm{H}^2_{\mathfrak{F}}(G,\underline{E}^{\times}) \cong \cap_{H \in \mathfrak{F}} \mathrm{Br}(E^{\scriptscriptstyle H}/K)\subseteq \mathrm{Br}(E/K).\]
\end{proposition}
\begin{proof} From the preceding discussion, we conclude: $$\mathrm{Br}(E^{\scriptscriptstyle H}/K)= \mathrm{Ker} \ \Big(i_{\scriptscriptstyle H}^2: \mathrm{H}^2(G,E^{\times}) \rightarrow \mathrm{H}^2(H,E^{\times})\Big),$$ where $i_{\scriptscriptstyle H}^2$ is induced by the inclusion $i_{\scriptscriptstyle H} : H \hookrightarrow G$. This implies: $$\cap_{H \in \mathfrak{F}} \mathrm{Br}(E^{\scriptscriptstyle H}/K)= \cap_{H \in \mathfrak{F}}\mathrm{Ker} \ \Big(i_{\scriptscriptstyle H}^2: \mathrm{H}^2(G,E^{\times}) \rightarrow \mathrm{H}^2(H,E^{\times})\Big).$$ Now, Lemma \ref{lemma: 2nd-bredon-galois} implies that
\[ \mathrm{H}^2_{\mathfrak{F}}(G,\underline{E}^{\times}) \cong \cap_{H \in \mathfrak{F}} \mathrm{Br}(E^{\scriptscriptstyle H}/K).  \]
\end{proof}
As a consequence, we are able to express the relative Brauer group of a finite seperable non-normal field extension as a Bredon-Galois cohomology group.
\begin{corollary}\label{cor: int brauer} Let $L/K$ be a finite separable extension and let $E$ be a field  such that $E/L$ is a finite extension and $E$ is Galois over $K$. Denote the Galois group of $E/K$ by $G$. Let $H$ be the subgroup of $G$ such that $E^{\scriptscriptstyle H}=L$. Then
\[
 \mathrm{H}^2_{\mathfrak{F}}(G,\underline{E}^{\times}) \cong \mathrm{Br}(L/K),
\]
 where $\mathfrak{F}$ is the smallest family of subgroups of $G$ that is closed under conjugation, closed under taking subgroups and contains $H$.
\end{corollary}
\begin{proof} First, let $\mathfrak{F}$ be the family containing just $H$ and the trivial subgroup. Then, the proposition states that $\mathrm{H}^2_{\mathfrak{F}}(G,\underline{E}^{\times}) \cong \mathrm{Br}(L/K)$. Now, using Corollary \ref{cor: intersection}, it is straightforward to check that enlarging $\mathfrak{F}$ with subgroups of $H$ and their conjugates does not change $\mathrm{H}^2_{\mathfrak{F}}(G,\underline{E}^{\times})$ inside $\mathrm{H}^2(G,{E}^{\times})$.
\end{proof}
As we mention in the proof, this corollary is also valid when $\mathfrak{F}$ consists only of $H$ and the trivial subgroup. The reason for making $\mathfrak{F}$ larger ensures that $\mathrm{H}^2_{\mathfrak{F}}(G,\underline{E}^{\times})$ has nicer functorial properties. We list these properties in the next lemma whose proof can be found in \cite[4.14, 5.4]{Hambleton} (see also \cite{ThevenazWebb}).

Let us establish the notation:
$$\mF\cap S=\{H\cap S \; |\; H\in \mF\},$$ for any given subgroup of $S$ of $G$.
\begin{lemma}
Let $\Gamma$ be a finite group,  $\mathfrak{F}$ be a family of subgroups of $\Gamma$ that is closed under conjugation and taking subgroups and  $M \in \mbox{mod-}\mathcal{O}_{\mathfrak{F}}(\Gamma)$ be a (cohomological) Mackey-functor. Denote by $\mathcal{O}\Gamma$ the orbit category with respect to the family that contains all subgroups of $\Gamma$. Then the functor:
\[ \mathrm{H}^n_{\mathfrak{F}\cap -}(-,M): \mathcal{O}\Gamma \rightarrow \mathfrak{Ab} : \Gamma/S \mapsto \mathrm{H}^n_{\mathfrak{F}\cap S}(S,M) \]
is again a (cohomological) Mackey functor, for each $n$.  This implies that for all $n \in \mathbb{N}$, all $g \in G$ and  all subgroups $L,P$ of $G$ such that $L\subseteq P$, we have a restriction map:
\[ \mathrm{res}_L^P: \mathrm{H}^n_{\mathfrak{F}\cap P}(P,\underline{E}^{\times}) \rightarrow  \mathrm{H}^n_{\mathfrak{F}\cap L}(L,\underline{E}^{\times}), \]
a transfer map:
\[ \mathrm{tr}_L^P : \mathrm{H}^n_{\mathfrak{F}\cap L }(L,\underline{E}^{\times}) \rightarrow  \mathrm{H}^n_{\mathfrak{F}\cap P}(P,\underline{E}^{\times}) \]
and a conjugation map:
\[ c_g(P): \mathrm{H}^n_{\mathfrak{F}\cap P }(P,\underline{E}^{\times}) \rightarrow  \mathrm{H}^n_{\mathfrak{F}\cap {}^gP}({}^gP,\underline{E}^{\times}) \]
such that
 \[ \mathrm{res}_P^G \circ \mathrm{tr}_L^G = \sum_{x \in [P\setminus G/L]}  c_x(P^x) \circ \mathrm{tr}_{P^{x}\cap L}^{P^{x}} \ \circ \mathrm{res}_{P^x \cap L}^L. \]
In addition, if $M$ is a cohomological Mackey functor, we also have:
\[ \mathrm{tr}_L^P \circ \mathrm{res}_L^P= [P:L]id.\]
The restriction, transfer and conjugation maps are also natural, in the sense that they commute with connecting homomorphisms and maps induced by natural transformations between coefficients.
\end{lemma}

\indent In what follows, we shall use this Mackey functor structure on Bredon-Galois cohomology, together with Proposition \ref{prop: bredon-galois}, to study the relative Brauer group of a finite separable extension of fields.
\subsection{Two short exact sequences} \label{sec: short exact seq}
Throughout this section, we assume that $E/K$ is a finite Galois extension with Galois group $G$.
Note that fixed point functors are cohomological Mackey functors (see \cite{ThevenazWebb}). Thus, if $\mathfrak{F}$ is any family of subgroups of $G$ closed under conjugation and taking subgroup, we have $\mathrm{tr}_{\{e\}}^G \circ \mathrm{res}_{\{e\}}^G= |G|id$. Since $\mathrm{H}^{n}_{\mathfrak{F}\cap \{e\} }(\{e\},\underline{E}^{\times})=0$ for all $n \geq 1$, it follows that $|G|id$ is the zero map on $\mathrm{H}^{n}_{\mathfrak{F}}(G,\underline{E}^{\times})$ for all $n \geq 1$. We conclude that $\mathrm{H}^{n}_{\mathfrak{F}}(G,\underline{E}^{\times})$ is a torsion abelian group, and
\[ \mathrm{H}^{n}_{\mathfrak{F}}(G,\underline{E}^{\times})= \bigoplus_{\substack{p \ \mbox{\tiny prime} \\ p \ \mbox{\tiny divides} \ |G|}} \mathrm{H}^{n}_{\mathfrak{F}}(G,\underline{E}^{\times})_{(p)} \]
for all $n \geq 1$, where $\mathrm{H}^{n}_{\mathfrak{F}}(G,\underline{E}^{\times})_{(p)}$ is the $p$-primary component of $\mathrm{H}^{n}_{\mathfrak{F}}(G,\underline{E}^{\times})$.
\begin{proposition} \label{lemma: odd cohomology is zero} Let $\mathfrak{F}$ be any family of subgroups of $G$ that is closed under conjugation and taking subgroups and let $p$ be a prime. If the $p$-Sylow subgroups of $G$ are cyclic, then we have $\mathrm{H}^{2k+1}_{\mathfrak{F}}(G,\underline{E}^{\times})_{(p)}=0$ for each $k \in \mathbb{N}$.
\end{proposition}
\begin{proof}
Let $P$ be a $p$-Sylow subgroup of $G$ and consider the transfer and restriction maps $\mathrm{tr}_P^G$ and $\mathrm{res}_P^G$.  It follows that $$\mathrm{tr}_P^G \circ \mathrm{res}_P^G= [G:P]id,$$ which implies that $\mathrm{res}_P^G$ injects $\mathrm{H}^{\ast}_{\mathfrak{F}}(G,\underline{E}^{\times})_{(p)}$ into $\mathrm{H}^{\ast}_{\mathfrak{F}\cap P}(P,\underline{E}^{\times})$. Hence, it suffices to prove the lemma for cyclic $p$-groups.

So, we assume that $G=\mathbb{Z}_{p^n}$ for some $n$ and proceed by induction on $n$. If $n=0$, the statement of the lemma is trivial. Now, assume the statement is correct for $\mathbb{Z}_{p^{n-1}}$ and consider $G=\mathbb{Z}_{p^n}$. If $\mathfrak{F}$ just contains the trivial subgroup then $\mathrm{H}^{2k+1}_{\mathfrak{F}}(G,\underline{E}^{\times})=\mathrm{H}^{2k+1}(G,E^{\times})$.  Since cyclic groups have periodic cohomology with period $2$, it follows from Hilbert's Theorem $90$ that
$\mathrm{H}^{2k+1}(G,E^{\times})=0$ for each $k \in \mathbb{N}$.

Suppose that $\mathfrak{F}$ does not just contain the trivial subgroup. Since $G$ contains a unique subgroup $H$ of order $p$, it follows that every non-trivial subgroup in $\mathfrak{F}$ contains $H$. Also, $H \in \mathfrak{F}$ because $\mathfrak{F}$ is subgroup closed. Let $\mathfrak{H}$ be the family $\mathfrak{F}\setminus{\{e\}}$. By Corollary $4.5$ of \cite{Martinez}, we can conclude that $$\mathrm{H}^{2k+1}_{\mathfrak{F}}(G,\underline{E}^{\times})=\mathrm{H}^{2k+1}_{\mathfrak{H}}(G,\underline{E}^{\times}).$$
Note that the family $\overline{\mathfrak{H}}=\{S/H \ | \ S \in \mathfrak{H}\}$ of subgroups of $G/H$ is subgroup closed.
Since \[ \mathcal{O}_{\mathfrak{H}}G \rightarrow \mathcal{O}_{\overline{\mathfrak{H}}}(G/H) : G/S \mapsto (G/H)/(S/H) \]
is an isomorphism of categories, we are done by induction.
\end{proof}
\medskip
\begin{theorem} Let $F/E$ be a finite Galois extension with Galois group $N \subseteq \mathrm{Gal}(F/K)=\Gamma$ such that $F/K$ is also Galois, and let $\mathfrak{F}$ be a family of subgroups of $\Gamma$ that is closed under conjugation and taking subgroups. For a prime $p$, if the $p$-Sylow subgroups of $G$ are cyclic, then we have a short exact sequence:
\[ 0 \rightarrow  \cap_{H \in \mathfrak{F}} \mathrm{Br}(E^{\scriptscriptstyle H}/K)_{(p)} \rightarrow \cap_{H \in \mathfrak{F}} \mathrm{Br}(F^{\scriptscriptstyle H}/K)_{(p)} \rightarrow \lim_{H \in \mathfrak{F}} \Big(\cap_{S \in \mathfrak{F}} \mathrm{Br}(F^{\scriptscriptstyle S \cap HN}/E^{\scriptscriptstyle H})_{(p)} \Big)  \rightarrow 0. \]
If all Sylow subgroups of $G$ are cyclic, then
\[ 0 \rightarrow  \cap_{H \in \mathfrak{F}} \mathrm{Br}(E^{\scriptscriptstyle H}/K) \rightarrow \cap_{H \in \mathfrak{F}} \mathrm{Br}(F^{\scriptscriptstyle H}/K) \rightarrow \lim_{H \in \mathfrak{F}} \Big(\cap_{S \in \mathfrak{F}} \mathrm{Br}(F^{\scriptscriptstyle S \cap  HN}/E^{\scriptscriptstyle H}) \Big)  \rightarrow 0. \]

\end{theorem}
\begin{proof}
By Galois theory, it follows that $G=\Gamma/N$. We define the family $\overline{\mathfrak{F}}=\{HN/N  \ | \ H \in \mathfrak{F} \}$ of subgroups of $G$ and note that this family is closed under conjugation and taking subgroups. From Theorem $5.1$ in \cite{Martinez}, we obtain a spectral sequence:
\[ \mathrm{E}_2^{p,q}=\mathrm{H}^p_{\overline{\mathfrak{F}}}(G,\mathrm{H}^q_{\mathfrak{F \cap -}}(-,\underline{F}^{\times})) \Rightarrow \mathrm{H}^p_{\mathfrak{F}}(\Gamma,\underline{F}^{\times}), \]
with $$\mathrm{H}^q_{\mathfrak{F} \cap -}(-,\underline{F}^{\times}): \mathcal{O}_{\overline{\mathfrak{F}}}G \rightarrow \mathfrak{Ab}: (\Gamma/N)/(HN/N) \mapsto \mathrm{H}^q_{\mathfrak{F} \cap HN }(HN,\underline{F}^{\times}).$$ Using Lemma \ref{lemma: bredon-hilbert}, we conclude that $\mathrm{E}^{p,1}_2=0$ for all $p\geq 0$. A standard spectral sequence argument entails a 5-term exact sequence from which we deduce the exact sequence:
\[ 0 \rightarrow \mathrm{H}^2_{\overline{\mathfrak{F}}}(G,\mathrm{H}^0_{\mathfrak{F \cap -}}(-,\underline{F}^{\times})) \rightarrow \mathrm{H}^2_{\mathfrak{F}}(\Gamma,\underline{F}^{\times}) \rightarrow \mathrm{H}^0_{\overline{\mathfrak{F}}}(G,\mathrm{H}^2_{\mathfrak{F \cap -}}(-,\underline{F}^{\times})) \rightarrow \mathrm{H}^3_{\overline{\mathfrak{F}}}(G,\mathrm{H}^0_{\mathfrak{F \cap -}}(-,\underline{F}^{\times})). \]
Using the limit interpretation of 0-dimensional Bredon cohomology, it is straightforward to check that: $$\mathrm{H}^0_{\mathfrak{F} \cap HN}(HN,\underline{F}^{\times})=(F^{\times})^{HN}.$$
The fact that $F^N=E$ implies $\mathrm{H}^0_{\mathfrak{F \cap -}}(-,\underline{F}^{\times})=\underline{E}^{\times}$. By restricting the exact sequence above to the $p$-primary components and using Proposition \ref{lemma: odd cohomology is zero}, we obtain a short exact sequence:
\[ 0 \rightarrow \mathrm{H}^2_{\overline{\mathfrak{F}}}(G,\underline{E}^{\times})_{(p)} \rightarrow \mathrm{H}^2_{\mathfrak{F}}(\Gamma,\underline{F}^{\times})_{(p)} \rightarrow \mathrm{H}^0_{\overline{\mathfrak{F}}}(G,\mathrm{H}^2_{\mathfrak{F \cap -}}(-,\underline{F}^{\times}))_{(p)} \rightarrow 0. \]
Using Proposition \ref{prop: bredon-galois} and the limit-interpretation of 0-dimensional Bredon cohomology, the sequence transforms to:
\[ 0 \rightarrow  \cap_{H \in \mathfrak{F}} \mathrm{Br}(E^{\scriptscriptstyle H}/K)_{(p)} \rightarrow \cap_{H \in \mathfrak{F}} \mathrm{Br}(F^{\scriptscriptstyle H}/K)_{(p)} \rightarrow \lim_{H \in \mathfrak{F}} \Big(\cap_{S \in \mathfrak{F}} \mathrm{Br}(F^{\scriptscriptstyle S \cap HN}/E^{\scriptscriptstyle H})_{(p)} \Big)  \rightarrow 0. \]
The second statement of the theorem follows immediately from the first.
\end{proof}
\medskip
\begin{theorem} Let $\mathfrak{F}$ be any family of subgroups of $G$ that is closed conjugation and taking subgroups and let $p$ be a prime. If the $p$-Sylow subgroups of $G$ are cyclic, then there exists a short exact sequence:
\[ 0 \rightarrow  \cap_{H \in \mathfrak{F}} \mathrm{Br}(E^{\scriptscriptstyle H}/K)_{(p)} \rightarrow \mathrm{Br}(E/K)_{(p)} \rightarrow \lim_{H \in \mathfrak{F}}\mathrm{Br}(E/E^{\scriptscriptstyle H})_{(p)} \rightarrow 0. \]
If all Sylow subgroups of $G$ are cyclic, then
\[ 0 \rightarrow  \cap_{H \in \mathfrak{F}} \mathrm{Br}(E^{\scriptscriptstyle H}/K) \rightarrow \mathrm{Br}(E/K) \rightarrow \lim_{H \in \mathfrak{F}}\mathrm{Br}(E/E^{\scriptscriptstyle H}) \rightarrow 0 \]
is exact.
\end{theorem}
\begin{proof} Consider the spectral sequence in Theorem $6.1$ of \cite{Martinez}:
\[ \mathrm{E}_2^{p,q}=\mathrm{H}^p_{{\mathfrak{F}}}(G,\mathrm{H}^q(-,E^{\times})) \Rightarrow \mathrm{H}^p(G,E^{\times}), \]
with $$\mathrm{H}^q(-,E^{\times}): \mathcal{O}_{{\mathfrak{F}}}G \rightarrow \mathfrak{Ab}: G/H \mapsto \mathrm{H}^q (H,E^{\times}).$$
Following a similar argument as in the proof of the previous theorem finishes the claim.
\end{proof}
\subsection{Constructing non-zero elements in relative Brauer groups} \label{sec: non-zero}
Our approach in this section is an adaptation to the Bredon setting of methods used in \cite{FeinSaltmanSchacher} and \cite{FeinSchacher}.

Let us recall some terminology and result from discrete valuation theory.
\emph{A discrete valuation} of $v$ of a field $K$ is a surjective group homomorphism: $$v: K^{\times} \rightarrow \mathbb{Z}\mbox{\; such that \;} v(x+y)\geq \min\{v(x),v(y)\}$$ for all $x,y \in K^{\times}$. Let $L/K$ be a finite separable extension and let $\delta : K^{\times} \rightarrow \mathbb{Z}$ be a discrete valuation of $K$. An \emph{extension} of $\delta$ to $L$ is a discrete valuation: $$\theta: L^{\times} \rightarrow \mathbb{Z}\mbox{\; such that \;} \theta|{_{K^{\times}}}=\delta.$$ Let us clarify this definition. Denote the index of $\theta(K^{\times})$ in $\mathbb{Z}$ by $e_{\delta}^{\theta}$. This number is called the \emph{ramification index} of $\theta$ with respect to $\delta$. Hence, $\theta|{_{K^{\times}}} : K^{\times}\rightarrow e_{\delta}^{\theta}\mathbb{Z}$ is a surjective group homomorphism. By $\theta|{_{K^{\times}}}=\delta$, we mean that $\alpha \circ \theta|{_{K^{\times}}}= \delta$, where $\alpha$ is the isomorphism $e_{\delta}^{\theta}\mathbb{Z} \rightarrow \mathbb{Z}:  e_{\delta}^{\theta}n \mapsto n$. Note that if $x \in K^{\times}$ and $\delta(x)=n$, then $\theta(x)=e_{\delta}^{\theta}n$. \\
It can be shown that there exist only a finite number of extensions of $\delta$ to $L$. \\
\indent Let $E/K$ be a finite Galois extension, let $\delta : K^{\times} \rightarrow \mathbb{Z}$ be a {discrete valuation} of $K$ and let $\pi:  E^{\times} \rightarrow \mathbb{Z}$ be an extension of $\delta$ to $E$. Denote the Galois group of $E/K$ by $G$ and let $\Omega$ be the set of extensions of $\delta$ to $E$. The group $G$ acts on $\Omega$ as follows: for $\pi \in \Omega$ and $g \in G$,  $$(g\cdot \pi )(x)=\pi(g^{\scriptscriptstyle -1}(x))$$ for all $x \in E^{\times}.$ Is is easy to check that this defines a left group action of $G$ on $\Omega$. In fact, one can show that this action in transitive. Using the fact that $\Omega$ is a transitive $G$-set, one can also show that all the valuations in $\Omega$ have the same ramification index with respect to $\delta$. The isotropy group of $G$ at $\pi$ is called the \emph{decomposition group} of $\pi$ and is denoted by $G_{\pi}$.
The valuation $\delta$ is called \emph{unramified} in $E$ if $e_{\delta}^{\pi}=1$. Note that if $M$ is a field such that $E \supseteq M \supseteq K$ and $\delta$ is unramified in $E$ , then the valuation $\pi|{_M}$ of $M$ is also unramified in $E$.

Let $E/K$ be a finite Galois extension and set $G=\mathrm{Gal}(E/K)$. Let $\mathfrak{F}$ be a family of subgroups of $G$ that is closed under conjugation and closed under taking subgroups. Let $\delta : K^{\times} \rightarrow \mathbb{Z}$ be a discrete valuation of $K$ and let $\pi:  E^{\times} \rightarrow \mathbb{Z}$ be an extension of $\delta$ to $E$. Since $G_{\pi}$ is the isotropy of $\pi$, the map $\pi: E^{\times} \rightarrow \mathbb{Z}$ is a $G_{\pi}$-module homomorphism. Hence, $\pi$ induces a natural transformation $\underline{\pi}$ between the fixed point functors $\underline{E}^{\times}$ and $\underline{\mathbb{Z}}$ in $\mbox{Mod-}\mathcal{O}_{\mathfrak{F}\cap G_{\pi}}(G_{\pi})$. We define:  $$\underline{\pi}^{\ast}:\mathrm{H}^2_{\mathfrak{F}\cap G_{\pi}}(G_{\pi},\underline{E}^{\times}) \rightarrow \mathrm{H}^2_{\mathfrak{F}\cap G_{\pi}}(G_{\pi},\underline{\mathbb{Z}})$$ to be the map induced by $\underline{\pi}$. Denote   $$\mathfrak{F}(G_{\pi}):=\{f \in \mathrm{Hom}(G_{\pi},\mathbb{Q}/\mathbb{Z}) \ | \ f(H\cap G_{\pi})=0, \ \forall \ H \in \mathfrak{F} \}.$$
Note that, by Corollary \ref{cor: finite group Q}, we have an isomorphism: $$\mathrm{H}^2_{\mathfrak{F}\cap G_{\pi}}(G_{\pi},\underline{\mathbb{Z}}) \xrightarrow{\cong} \mathfrak{F}(G_{\pi}).$$
The \emph{character map} $\chi_{\delta} : \mathrm{H}^2_{\mathfrak{F}}(G,\underline{E}^{\times}) \rightarrow \mathfrak{F}(G_{\pi})$ is defined by the following composition of maps:
\[ \mathrm{H}^2_{\mathfrak{F}}(G,\underline{E}^{\times}) \xrightarrow{\mathrm{res}_{G_{\pi}}^{G}} \mathrm{H}^2_{\mathfrak{F}\cap G_{\pi}}(G_{\pi},\underline{E}^{\times}) \xrightarrow{\underline{\pi}^{\ast}} \mathrm{H}^2_{\mathfrak{F}\cap G_{\pi}}(G_{\pi},\underline{\mathbb{Z}}) \xrightarrow{\cong} \mathfrak{F}(G_{\pi}). \]

\indent The following lemma reveals a key property of the character map.
\medskip
\begin{lemma} \label{lemma: transfer lemma}Let $E \supseteq M \supseteq K$ be a tower of fields such that $E$ is a finite Galois extension of $K$ with Galois group G and suppose $P$ is the subgroup of $G$ such that $E^{P}=M$. Let $\delta$ be a discrete valuation of $K$ that is unramified in $E$, and suppose that $\theta_1,\ldots,\theta_n$ are all the possible extensions of $\delta$ to $M$.
Let $\{\pi=\pi_1,\ldots,\pi_n\}$ be a set of discrete valuations of $E$ such that $\pi_i$ extends $\theta_i$. Finally, let $\{e=x_1,\ldots,x_n\}$ be a set of elements in $G$ such that $x_i \cdot \pi_i = \pi$. Let $\mathfrak{F}$ be a family of subgroups of $G$ that is closed under conjugations and taking subgroups. Now consider the character maps:
\[ \chi_{\delta}: \mathrm{H}^2_{\mathfrak{F}}(G,\underline{E}^{\times}) \rightarrow \mathfrak{F}(G_{\pi})  \]
and
\[ \chi_{\theta_i}: \mathrm{H}^2_{\mathfrak{F}\cap P}(P,\underline{E}^{\times}) \rightarrow \chi_{\mathfrak{F}\cap P}(P_{\pi_i})  \]
for all $i \in \{1,\ldots,n\}$. For every $\alpha \in \mathrm{H}^2_{\mathfrak{F}\cap P}(P,\underline{E}^{\times})$, we have:
\[ \chi_{\delta} \circ \mathrm{tr}_{P}^G(\alpha) = \sum_{i=1}^n c_{x_i}(G_{\pi_i}) \circ \mathrm{tr}_{P_{\pi_i}}^{G_{\pi_i}} \circ \chi_{\theta_i}(\alpha). \]
\end{lemma}
\begin{proof}
Denote the set of extension of $\delta$ to $E$ by $\Omega$ and denote the set of extensions of $\theta_i$ to $E$ by $\Omega_i$ for all $i \in \{1,\ldots,n\}$. Then $\pi_i \in \Omega_i \subseteq \Omega$ for all $i$. We claim that first, $\Omega = \coprod_{i=1}^n \Omega_i $; secondly, $G_{\pi}^{x_i}\cap P=P_{\pi_i}$; and finally, $\{x_1,\ldots,x_n\}$ is a set of representatives of the double cosets $G_{\pi} \setminus G / P$.

The first two claims are easy to verify. Let us prove the third claim. Let $g \in G$ and consider $g^{\scriptscriptstyle -1}\cdot \pi \in \Omega$. There exists a unique $i \in \{1,\ldots,n\}$ such that $g^{\scriptscriptstyle -1}\cdot \pi \in \Omega_i$.  Fix this $i$. Since $\pi_i \in \Omega_i$ and $P$ acts transitively on $\Omega_i$, we can find $p \in P$ such that $p\cdot \pi_i = g^{\scriptscriptstyle -1} \cdot \pi$.  Since $x_i \cdot \pi_i = \pi$, we conclude that $(gpx_i^{\scriptscriptstyle -1})\cdot \pi = \pi$. This implies that we can find an element $y \in G_{\pi}$ such that $gpx_i^{\scriptscriptstyle -1}=y$. To summarize, for all $g \in G$, we can find  $x \in G_\pi$ and $p \in P$ such that $xgp=x_i$, where $i$ is the unique element of $\{1,\ldots,n\}$ such that $g^{\scriptscriptstyle -1}\cdot \pi \in \Omega_i$. Now, assume that $x_i=xx_jp$ for some $x \in G_{\pi}$, some $p \in P$ and some $i,j \in \{1,\ldots,n\}$. One easily verifies that this implies that $p\cdot \pi_i=\pi_j$. Since $p \in \mathrm{Gal}(E/M)$, this implies that for all $m \in M$, $$\theta_j(m)=\pi_j(m)=\pi_i(p^{\scriptscriptstyle -1}(m))=\pi_i(m)=\theta_i(m).$$ Therefore, $\theta_i=\theta_j$ which implies that $i=j$. This proves our claim. \\
A priori, we have the relation:
 \[ \mathrm{res}_{G_{\pi}}^{G} \circ \mathrm{tr}_P^G = \sum_{x \in [G_{\pi}\setminus G/L]}  c_x(G_{\pi}^x) \circ \mathrm{tr}_{G_{\pi}^{x}\cap P}^{G_{\pi}^{x}} \ \circ \mathrm{res}_{G_{\pi}^x \cap P}^P. \]
Using our claims, we can rewrite this formula as:
\[ \mathrm{res}^G_{G_{\pi}} \circ \mathrm{tr}_P^G = \sum_{i=1}^n c_{x_i}(G_{\pi_i})  \circ \mathrm{tr}_{P_{\pi_i}}^{G_{\pi_i}} \circ \mathrm{res}_{P_{\pi_i}}^P. \]
The lemma now follows from the definition of the character maps (note that the isomorphism of the last map in the definition is given by a connecting homomorphism) and the naturality of transfer, restriction and conjugation.
\end{proof}
\begin{definition} \rm Let $k \in \mathbb{N}\cup \{\infty\}$. We say that an extension of fields $E/K$ is $k$-\emph{admissible} if it is a finite Galois extension with Galois group $G$ such that for every $\sigma \in G\setminus \{e\}$, there exist at least $k$ discrete valuations $\pi$ of $E$ such that $G_{\pi}=\langle\sigma\rangle$ and such that $\pi|{_K}$ is unramified in $E$.
\end{definition}
We can now prove the main theorem of this section.
\begin{theorem} \label{th: relative brauer}
Let $E/K$ be a $k$-admissible extension with Galois group $G$ and let $\mathfrak{F}$ be any family of subgroups of $G$ that is closed under conjugation and taking subgroups. If there exists an element $\sigma \in G$ such that $\mathfrak{F}( \langle \sigma \rangle)$ is non-zero, then $\cap_{H \in \mathfrak{F}} \mathrm{Br}(E^{\scriptscriptstyle H}/K)$ contains at least $k$ non-zero elements.
\end{theorem}
\begin{proof}
Let $P$ be the subgroup of $G$ generated by $\sigma$ and denote $E^{P}$ by $M$. Because $E/K$ is a $k$-admissible extension, there exist discrete valuations ${}^{1}\pi,\ldots,{}^{k}\pi $ of $E$ such that for all $s \in \{1\ldots,k\}$, ${}^{s}\pi|{_K}$ is unramified in $E$ and $G_{{}^{s}\pi}=P$. We set ${}^{s}\delta={}^{s}\pi|{_K}$ and ${}^{s}\theta={}^{s}\pi|{_M}$  for each $s \in \{1\ldots,k\}$ and observe that $P=P_{{}^{s}\pi}$.
For each $s \in \{1,\ldots,k\}$,  let $\{{}^{s}\theta_1={}^{s}\theta,\ldots,{}^{s}\theta_{n_s}\}$ be all the possible extensions of ${}^{s}\delta$ to $M$ and denote $n=n_1$.
 By the approximation theorem (see \cite{ZariskiSamuel} Ch. VI, Th. $18$), there exist elements: $$m_1,\ldots,m_k\in M^{\times}\mbox{\; so that \;}  {}^{s}\theta_i(m_l)=\delta_{i1}\delta_{ls}$$ for all $s,l \in \{1\ldots,k\}$ and all $i \in \{1,\ldots,n_s\}$. Now,  fix a non-zero element ${f} \in \mathfrak{F}(P)$. Let $\tilde{f}\in \mathrm{H}^2_{\mathfrak{F}\cap P}(P,\underline{\mathbb{Z}})$ be the element that corresponds to ${f}$ under the isomorphism: $$\mathrm{H}^2_{\mathfrak{F}\cap P}(P,\underline{\mathbb{Z}})\cong \mathfrak{F}(P).$$ For each $l \in \{1\ldots,k\}$, the $P$-module homomorphism: $$\Delta_l: \mathbb{Z} \rightarrow E^{\times} : 1 \mapsto m_l$$ induces a natural transformation $\underline{\Delta_l}: \underline{\mathbb{Z}} \rightarrow \underline{E}^{\times}$. Moreover, for each $l \in \{1\ldots,k\}$,  the morphism $\underline{\Delta_l}$ induces a map: $$\underline{\Delta_l}^{\ast}:  \mathrm{H}^2_{\mathfrak{F}\cap P}(P,\underline{\mathbb{Z}}) \rightarrow \mathrm{H}^2_{\mathfrak{F}\cap P}(P,\underline{E}^{\times}).$$ Denote $\underline{\Delta_l}^{\ast}(\tilde{f})$ by $\alpha_l$ for each $\in \{1\ldots,k\}$. We claim that $\{ \mathrm{tr}_{P}^G(\alpha_1),\ldots, \mathrm{tr}_{P}^G(\alpha_k)\}$
is a set of $k$ distinct, non-zero elements of $\mathrm{H}^2_{\mathfrak{F}}(G,\underline{E}^{\times})$.

Let us show that $\mathrm{tr}_{P}^G(\alpha_1)$ is non-zero and does not equal $\mathrm{tr}_{P}^G(\alpha_2)$. The remaining cases are of course similar, so the claim will follow. Let $\{{}^{1}\pi={}^{1}\pi_1,\ldots,{}^1\pi_n\}$ be a set of discrete valuations of $E$ such that ${}^{1}\pi_i$ extends ${}^{1}\theta_i$ and let $\{e=x_1,\ldots,x_n\}$ be a set of elements in $G$ such that $x_i \cdot {}^{1}\pi_i = {}^{1}\pi$. It follows from Lemma \ref{lemma: transfer lemma} that
\begin{equation} \label{eq: transfer 1} \chi_{{}^{1}\delta}(\mathrm{tr}_{P}^G(\alpha_1)) = \sum_{i=1}^n c_{x_i}(G_{{}^{1}\pi_i}) \circ \mathrm{tr}_{P_{{}^{1}\pi_i}}^{G_{{}^{1}\pi_i}} \circ \chi_{{}^{1}\theta_i}(\alpha_1) \end{equation}
and
\begin{equation} \label{eq: transfer 2} \chi_{{}^{1}\delta}(\mathrm{tr}_{P}^G(\alpha_2)) = \sum_{i=1}^n c_{x_i}(G_{{}^{1}\pi_i}) \circ \mathrm{tr}_{P_{{}^{1}\pi_i}}^{G_{{}^{1}\pi_i}} \circ \chi_{{}^{1}\theta_i}(\alpha_2). \end{equation}
By naturality, we have: $$\underline{{}^{1}\pi_i}^{\ast} \circ \mathrm{res}_{P_{{}^{1}\pi_i}}^P(\alpha_l)=\underline{{}^{1}\pi_i \circ \Delta_l}^{\ast} \circ  \mathrm{res}_{P_{{}^{1}\pi_i}}^P(\tilde{f})$$ for $l=1,2$. Observe that the valuations in $\{{}^{1}\theta_1={}^{1}\theta,\ldots,{}^{1}\theta_n\}$ are unramified in $E$, which implies that ${}^{1}\pi_i(m_l)=\delta_{1l}\delta_{i1}$. Using this observation, one checks that $\underline{{}^{1}\pi_i \circ \Delta_l}^{\ast}=\delta_{1l}\delta_{i1}$. Then it follows from (\ref{eq: transfer 1}) and (\ref{eq: transfer 2}) that
\begin{equation*} \chi_{{}^{1}\delta}(\mathrm{tr}_{P}^G(\alpha_1)) = f\neq 0 \end{equation*}
and
\begin{equation*}  \chi_{{}^{1}\delta}(\mathrm{tr}_{P}^G(\alpha_2)) = 0. \end{equation*}
This proves our claim. The theorem now follows from Proposition \ref{prop: bredon-galois}.
\end{proof}
Using Theorem \ref{th: relative brauer}, we reprove a result first obtained in \cite{FeinKantorSchacher}.
\begin{corollary}[Fein-Kantor-Schacher] \label{cor: fein}Let $L \supsetneq K$ be a finite separable extension of global fields, then $\mathrm{Br}(L/K)$ is infinite.
\end{corollary}
\begin{proof}
Consider a field $E$ containing $L$ such that $E/K$ is finite Galois with Galois group $G$. Let $H$ be the subgroup of $G$ such that $E^{\scriptscriptstyle H}=L$ and let $\mathfrak{F}$ is the smallest family of subgroups of $G$ that is closed under conjugation, closed under taking subgroups and contains $H$. It follows from the \emph{Chebotarev density theorem} (see \cite{FriedJarden}) that the extension $E/K$ is $\infty$-admissible.

Now, let $G$ act on $G/H$ by multiplication on the left. Since this action is transitive and $H \neq G$, by the lemma below, there exists an element in $\sigma \in G\setminus \{e\}$ of prime power order that acts without fixed points on $G/H$. This implies that $\sigma \notin S$ for all $S \in \mathfrak{F}$.
Denote $P=\langle\sigma\rangle$ and suppose $|P|=p^n$, for some prime $p$ and some $n \in \mathbb{N}_0$.
We claim that $\mathfrak{F}(P)$ is non-zero.

To prove this claim, it suffices to show that the subgroup of $P$ generated by $\mathfrak{F}\cap P$ does not equal $P$. Suppose that $P$ is generated by $\mathfrak{F}\cap P$.
This implies that there exist elements $\sigma^{i_1},\ldots,\sigma^{i_s}$ such that for each $j \in \{1,\ldots,s\}$ we have $\sigma^{i_j} \in P\cap H^{g_j}$ for some $g_j \in G$ and such that $\prod_{j=1}^s\sigma^{i_j}=\sigma$. It follows that $\sigma^{\sum_{j=1}^s i_j}=\sigma$, and hence, $\sum_{j=1}^s i_j=1 \mod p^n$. So, there must exists a $j$ such that $i_j$ is not divisible by $p$. For this $j$, we have $\gcd\{i_j,p^n\}=1$. It follows that $\sigma^{i_j}$ is a generator of $P$. Thus, there exists  $r \in \mathbb{Z}$ so that $(\sigma^{i_j})^r=\sigma$. Since $\sigma^{i_j} \in H^{g_{j}}$, this implies  $\sigma \in H^{g_{j}}$, which is a contradiction. This proves the claim.

Following the argument of the proof of Theorem \ref{th: relative brauer} shows that $\mathrm{H}^2_{\mathfrak{F}}(G,\underline{E}^{\times})$ is infinite. Corollary \ref{cor: int brauer} then implies that $\mathrm{Br}(L/K)$ is infinite.
\end{proof}
\begin{lemma}[see \cite{FeinKantorSchacher}] \label{lemma: finite simple group lemma} If a finite group $G$ acts transitively on a set $X$ containing more than one element, then there exists an element of prime power order in $G$ that acts without fixed points.
\end{lemma}

\end{document}